\documentclass[a4paper]{amsart}

\usepackage{amsmath, amssymb, amsthm}
\usepackage{lineno
}
\usepackage{enumerate}

\usepackage{mathdots}
\usepackage{bm}

\usepackage{comment}
\usepackage{float}

\newtheorem{theorem}{Theorem}[section]

\newtheorem{cor}[theorem]{Corollary}

\newtheorem*{Thm1}{Theorem 3.6}
\newtheorem*{cor1}{Corollary 4.11}

\newtheorem{prop}[theorem]{Proposition}
\newtheorem{lem}[theorem]{Lemma}

\newtheorem{example}[theorem]{Example}

\usepackage{graphicx}

\usepackage{tikz}
\usetikzlibrary{positioning, shapes.callouts}
\usepackage{pxpgfmark}

\allowdisplaybreaks

\begin{document}
	
		\title[Noncompact Iwasawa factorization and affine spheres]{Noncompact Iwasawa factorization and translationally equivariant hyperbolic affine spheres}
		\author{Tadashi Udagawa}
	    \date{}
	    \maketitle
		
		\begin{abstract}
		 We establish the noncompact Iwasawa factorization for a Delaunay-type potential associated with affine spheres on the complex plane away from countably many lines. Using the DPW method, we give an explicit description of the factorization in terms of Weierstrass elliptic functions via a reduction to a linear system related to the Tzitz\'eica equation. As an application, we construct explicit translationally equivariant hyperbolic affine spheres and classify them according to their slice curves. In particular, we show that every such affine sphere is equiaffinely equivalent to one whose slice curve is a circle, hyperbola, or parabola, consistent with the Calabi correspondence between hyperbolic affine spheres and proper convex cones.
		\end{abstract}
		\vspace{10pt}


	
\section{Introduction}
The Iwasawa factorization of loop groups plays a central role in the DPW method for constructing harmonic maps into symmetric spaces, introduced by J. Dorfmeister, F. Pedit, and H. Wu \cite{DPW1998}. Given a \(\Lambda \mathfrak{sl}_n\mathbb{C}\)-valued 1-form \(\xi\) on \(\mathbb{C}\), the factorization of a solution \(\phi\) to \(d\phi=\phi\xi\) yields a lift of a harmonic map and is also used to construct constant mean curvature (CMC) surfaces. A key difficulty is that the Iwasawa factorization is not globally defined in general: it is global in the compact case \cite{DPW1998}, but only local near the identity element in the noncompact case \cite{BD2001}. Concerning existence results, Dorfmeister, M. Guest, and W. Rossman established it for \(n=2\) Smyth-type potentials \cite{DGR2010}, and the factorization was described using Bessel functions in \cite{U2024}. For \(n=2\) Delaunay-type potentials, Brander, Rossman, and Schmitt \cite{BRS2010} proved local existence near \(t=0\), while a global extension to \(\mathbb{C}\) away from countably many lines, together with an explicit description in terms of the Lamé equation, was obtained in \cite{U2026}.\vskip\baselineskip

In this paper, we extend these results to the \(n=3\) case of the Delaunay-type potential
\begin{equation}
	\xi =
	\begin{pmatrix}
		m & \lambda^{-1}ia^{-1} & \lambda ia^2 \\
		\lambda ia^{-1} & 0 & \lambda^{-1}ia^{-1} \\
		\lambda^{-1}ia^2 & \lambda ia^{-1} & -m
	\end{pmatrix} dt,\ \ \ \ \ a>0,\ m \in \mathbb{R}.
	\nonumber
\end{equation}
We show that the corresponding \(\phi\) admits an Iwasawa factorization on \(\mathbb{C}\) away from countably many lines (Theorem \ref{thm3.4}), and provide an explicit description in terms of the Weierstrass functions \(\wp\), \(\zeta\), and \(\sigma\) (Propositions \ref{prop4.3} and \ref{prop4.5}). A key step is that the existence of the Iwasawa factorization reduces to the solvability of a linear differential equation of the form
\begin{equation}
	(\tilde{\Phi}_+^{-1})_t =
	\tilde{\Phi}_+^{-1}\left(\begin{array}{ccc}
		-\frac{\wp_t}{\wp - M} & 0 & -2\lambda ia^2 \\
		-2\lambda ia^{-1} & 0 & 0 \\
		0 & -2\lambda ia^{-1} & \frac{\wp_t}{\wp - M}
	\end{array}\right),\ \ \ \ \ \tilde{\Phi}_+^{-1}(0,\lambda) = I_3, \tag{A}
\end{equation}
where \(M = \frac{1}{3}(m^2 - a^4 - 2a^{-2})\). Here \(\wp\) is the Weierstrass elliptic function. In the DPW method for hyperbolic affine spheres, the Gauss-Codazzi equation associated to \(\xi\) reduces to the Tzitz\'eica equation
\begin{equation}
	u_{t\overline{t}} = a^{-2}e^u - a^4e^{-2u}.
	\nonumber
\end{equation}
Moreover, by the reality condition, the solution can be expressed in terms of the Weierstrass elliptic function \(\wp\). Here, \(\beta\) is a solution to \(\wp(\beta) = 2a^{-2} + M, \wp'(\beta) =  4a^{-2}m\) and \(\tilde{\Lambda}\) is the lattice defined in Section 3.1.
\begin{Thm1}
	Let \(M = \frac{1}{3}(m^2 - a^4 - 2a^{-2})\) and \(\tilde{\Phi}_+^{-1}\) a solution of (\ref{A'}) and define \(\mathcal{E}(x) = {\rm diag}(\varepsilon(x),1,\varepsilon(x))\), where
	\begin{equation}
		\varepsilon(x) = \left\{
		\begin{array}{cc}
			1 & {\rm if}\ \wp(x+\beta) - M > 0,\\
			-1 & {\rm if}\ \wp(x+\beta) - M < 0
		\end{array},\ \ \ \ \ x \in \mathbb{R}.\nonumber
		\right.
	\end{equation}
	Then \(\phi(t,\lambda)\) admits an Iwasawa factorization \(\phi = F\phi_+\) on a simply-connected subspace of
	\begin{equation}
		\mathbb{C} \backslash \{t = x+iy \in \mathbb{C}\ | \ x \in \tilde{\Lambda},\ y \in \mathbb{R}\}, \nonumber
	\end{equation}
	where \(F = F(t,\overline{t},\lambda) \in (\Lambda G)_{\sigma}\) and \(\phi_+ = \phi_+(t,\overline{t},\lambda) \in (\Lambda^+_B {\rm SL}_3 \mathbb{R})_{\sigma}\) are given by
	\begin{align}
		&F(t,{\bar t},\lambda) = \phi(iy,\lambda)\tilde{F}(x,\lambda)\mathcal{E}(x)g(x)^{\frac{1}{2}}
		\in (\Lambda G)_{\sigma}
		, \nonumber\\ 
		&\phi_+(t,{\bar t},\lambda) = 
		g(x)^{-\frac{1}{2}}\mathcal{E}(x)\tilde{\Phi}_+(x,\lambda) 
		\in (\Lambda^+_B {\rm SL}_3 \mathbb{C})_{\sigma}
		. \nonumber
	\end{align}
	\normalsize
\end{Thm1}\vskip\baselineskip

As an application, we obtain explicit descriptions of translationally equivariant hyperbolic affine spheres and classify them according to their slice curves, which turn out to be circles, hyperbolas, or parabolas (Theorem \ref{thm4.6}). Although explicit parametrizations of affine spheres are generally difficult to obtain, such structures have been studied via Wang's equation (Tzitz\'eica equation) and convex cone geometry, particularly in relation to self-associated cones arising from \(S^1\)-families of cubic differentials (Z. Lin, E. Wang \cite{LW2016}; U. Simon, C. P. Wang \cite{SW1993}). In \cite{H2014}, \cite{H2022}, R. Hildebrand classified such cones and obtained explicit isothermal parametrizations in terms of Painlev\'e transcendents, and Lin, Wang \cite{LW2016} further computed isothermal parametrizations for Hildebrand's examples. In contrast, our approach is based on the DPW method and a detailed analysis of the noncompact Iwasawa factorization associated with a Delaunay-type potential.\vskip\baselineskip

As a corollary of Theorem \ref{thm4.6}, we obtain a classification of translationally equivariant hyperbolic affine spheres according to their slice curves (Corollary \ref{cor4.7}):
\begin{cor1}
	Every translationally invariant solution of the Tzitz\'eica equation gives rise to a hyperbolic affine sphere equiaffinely equivalent to an affine sphere whose slice curve is a circle, a hyperbola, or a parabola.
\end{cor1}\vskip\baselineskip

This classification is consistent with the Calabi correspondence between hyperbolic affine spheres and proper convex cones in \(\mathbb{R}^3\). Up to equiaffine equivalence, the slice curves arising from translationally invariant solutions of the Tzitz\'eica equation are circles, hyperbolas, or parabolas.\vskip\baselineskip

The organization of this paper is as follows. In Section 2, we review affine geometry and the Tzitz\'eica equation, together with the DPW construction for hyperbolic affine spheres. In Section 3.1, we show that the Delaunay-type potential induces a solution of the Tzitz\'eica equation and yields a translationally equivariant hyperbolic affine sphere. In Section 3.2, we reduce the existence of the Iwasawa factorization to a differential equation and prove its existence on \(\mathbb{C}\) away from countably many lines (Theorem \ref{thm3.4}). In Section 4.1, we derive a scalar reduction and express solutions in terms of the Weierstrass functions \(\wp\), \(\zeta\), and \(\sigma\). In Section 4.2, we give explicit descriptions of translationally equivariant hyperbolic affine spheres (Propositions \ref{prop4.3} and \ref{prop4.5}), and classify them according to the sign of \(M = \frac{1}{3}(m^2 - a^4 - 2a^{-2})\) (Theorem \ref{thm4.6}). As a corollary, these surfaces are classified by their slice curves, which are circles, hyperbolas, or parabolas (Corollary \ref{cor4.7}).

\section{Preliminaries}
\subsection{Affine geometry and Tzitz\'eica equation}
In this section, we briefly review the relationship between the Tzitz\'eica equation and hyperbolic affine spheres. Similar arguments were given by Wang \cite{F1991}. Starting from a solution of the Tzitz\'eica equation, we recall the associated family of flat connections and the construction of the corresponding hyperbolic affine sphere.\vskip\baselineskip

Let \(U \subset \mathbb{C}\) be a simply-connected domain and let \(w:\mathbb{C} \to \mathbb{R}\) be a solution to the Tzitz\'eica equation
\begin{equation*}
	w_{t\overline{t}} = e^w - e^{-2w}. 
\end{equation*}
It is well known that such a solution determines a hyperbolic affine sphere. \vskip\baselineskip

Define the \(\mathfrak{sl}_3\mathbb C\)-valued 1-form
\begin{equation*}
	\alpha = \left(\begin{array}{ccc}
		\frac{w_t}{2} & \lambda^{-1}e^{\frac{w}{2}} & 0 \\
		0 & 0 & \lambda^{-1}e^{\frac{w}{2}} \\
		\lambda^{-1}e^{-w}  & 0 & -\frac{w_t}{2}
	\end{array}\right)dt + \left(\begin{array}{ccc}
		-\frac{w_{\overline{t}}}{2} & 0 & \lambda e^{-w} \\
		\lambda e^{\frac{w}{2}} & 0 & 0 \\
		0 & \lambda e^{\frac{w}{2}}  & \frac{w_{\overline{t}}}{2}
	\end{array}\right)d\overline{t}.
\end{equation*}

The Tzitz\'eica equation is equivalent to the flatness condition \(d\alpha + \alpha \wedge \alpha = 0\). Let \(F:\mathbb{C} \to {\rm SL}_3 \mathbb{C}\) be the solution of \(\alpha = F^{-1}dF,\ F(0)=I_3\) and define \(f = TF\begin{pmatrix}
	0 \\ 1 \\ 0
\end{pmatrix}\), where \(\omega^3 = 1\) and
\begin{equation}
	T = \frac{1}{\sqrt{3}}\left(\begin{array}{ccc}
		1 & 1 & 1 \\
		\omega & 1 & \omega^2 \\
		\omega^2 & 1 & \omega
	\end{array}\right). \nonumber
\end{equation}
Then \(f\) is an affine sphere.

\begin{prop}\label{prop2.1}
	\(f = f(t,\overline{t},\lambda):U \rightarrow \mathbb{R}^3\) is a hyperbolic affine sphere for all \(\lambda \in S^1\), with affine metric
	\begin{equation*}
		h(\partial_t,\partial_t) = h(\partial_{\overline{t}},\partial_{\overline{t}}) = 0, h(\partial_t,\partial_{\overline{t}}) = e^w
	\end{equation*}
	and affine normal vector field \(\xi = -f\).
\end{prop}
\begin{proof}
	Let
	\begin{equation}
		\Delta = \left(\begin{array}{ccc}
			0 & 0 & 1 \\
			0 & 1 & 0 \\
			1 & 0 & 0
		\end{array}\right). \nonumber
	\end{equation}
	From \(\Delta\overline{\alpha(\lambda)}\Delta = \alpha(\lambda)\) for all \(\lambda \in S^1\), we have \(\overline{F} = \Delta F \Delta\) and thus,
	\begin{align}
		\overline{TF} = 
		\overline{T}
		\Delta F \Delta
		= TF\Delta. \nonumber
	\end{align}
	Hence, \(f\) takes values in \(\mathbb{R}^3\). Moreover,
	\begin{equation*}
		f_t = \lambda^{-1}e^{\frac{w}{2}}TF\begin{pmatrix}
			1 \\ 0 \\ 0
		\end{pmatrix},\ \ \ \ \ f_{\overline{t}} = \lambda e^{\frac{w}{2}}TF\begin{pmatrix}
			0 \\ 0 \\ 1
		\end{pmatrix}, 
	\end{equation*}
	and hence \(f:U \to \mathbb{R}^3\) is an immersion. Using \(dF=F\alpha\), we obtain
	\begin{align*}
		&f_{tt} = w_tf_t + \lambda^{-3}e^{-w}f_{\overline{t}},\ \ \ \ \  f_{\overline{t}\overline{t}} = \lambda^3e^{-w}f_t + w_{\overline{t}}f_{\overline{t}},\ \ \ \ \ f_{t\overline{t}} = e^wf.
	\end{align*}
	These are precisely the structure equations of a hyperbolic affine sphere with affine normal vector field \(\xi=-f\) and affine metric
	\begin{equation*}
		h(\partial_t,\partial_t)
		=
		h(\partial_{\overline{t}},\partial_{\overline{t}})
		=
		0,
		\qquad
		h(\partial_t,\partial_{\overline{t}})
		=
		e^w.
	\end{equation*}
	This completes the proof.
\end{proof}\vskip\baselineskip

The following example corresponds to the trivial solution of the Tzitz\'eica equation.
\begin{example}\label{ex1}
	Let \(w=0\) on \(\mathbb{C}\). Then
	\begin{equation}
		\alpha = \lambda^{-1}\left(\begin{array}{ccc}
			0 & 1 & 0 \\
			0 & 0 & 1 \\
			1 & 0 & 0 
		\end{array}\right)dt + \lambda \left(\begin{array}{ccc}
			0 & 0 & 1\\
			1 & 0 & 0 \\
			0 & 1 & 0
		\end{array}\right)d\overline{t}. \nonumber
	\end{equation}
	The solution of \(F^{-1}dF=\alpha\) is given by
	\begin{align}
		F &= \exp{\left(\lambda^{-1}t\left(\begin{array}{ccc}
				0 & 1 & 0 \\
				0 & 0 & 1 \\
				1 & 0 & 0 
			\end{array}\right) + \lambda \overline{t}\left(\begin{array}{ccc}
				0 & 0 & 1\\
				1 & 0 & 0 \\
				0 & 1 & 0
			\end{array}\right)\right)} \nonumber\\
		&= T^{-1}\left(\begin{array}{ccc}
			e^{\lambda^{-1}t + \lambda\overline{t}} & 0 & 0 \\
			0 & e^{\lambda^{-1}\omega t + \lambda\omega^2\overline{t}} & 0 \\
			0 & 0 & e^{\lambda^{-1}\omega^2t + \lambda\omega\overline{t}}
		\end{array}\right)T. \nonumber
	\end{align}
	The corresponding affine sphere \(f = f(t,\overline{t},\lambda):\mathbb{C} \rightarrow \mathbb{R}^3\) is given by
	\begin{equation}
		f = TF\begin{pmatrix}
			0 \\ 1 \\ 0
		\end{pmatrix} = \frac{1}{\sqrt{3}}\begin{pmatrix}
			e^{\lambda^{-1}t + \lambda\overline{t}} \\
			e^{\lambda^{-1}\omega t + \lambda\omega^2\overline{t}} \\
			e^{\lambda^{-1}\omega^2t + \lambda\omega\overline{t}}
		\end{pmatrix}. \nonumber
	\end{equation}
	Writing \(f = \begin{pmatrix}
		x_1 \\ x_2 \\ x_3
	\end{pmatrix}\), we obtain \(x_1x_2x_3 = \frac{\sqrt{3}}{9}\). This is the hyperbolic affine sphere introduced by Calabi. \qed
\end{example}

\subsection{DPW method}
In this section, we review the DPW construction for hyperbolic affine spheres. The DPW approach in affine differential geometry was studied by Dorfmeister and Eitner \cite{DE2001} and by Dorfmeister and Ma \cite{DM2016}. We begin by introducing the relevant loop groups. Let
\begin{equation}
	\sigma(g) = \Delta (g^{-1})^t \Delta,\ \ \ \tau(g) = \Delta \overline{g} \Delta,\ \ \ \Delta = \left(\begin{array}{ccc}
		0& 0 & 1 \\
		0& 1 & 0 \\
		1 & 0 & 0
	\end{array}\right), \nonumber
\end{equation}
and
\begin{equation}
	G = \{g \in {\rm SL}_3 \mathbb{C}\ | \ \tau(g) = g\},\ \ \ \ \ \ K =\{g \in G\ | \ \sigma(g) = g\}. \nonumber
\end{equation}

\begin{lem}
	We have the decomposition 
	\begin{equation}
		K^{\mathbb{C}} = \{g \in {\rm SL}_3 \mathbb{C}\ | \ \sigma(g) = g\} = KB, \nonumber
	\end{equation}
	where
	\begin{equation}
		B = \left\{\exp{\left(\begin{array}{ccc}
				a & b & 0 \\
				\overline{b} & 0 & -b \\
				0 & -\overline{b} & -a
			\end{array}\right)}\ | \ a \in \mathbb{R},\ b \in \mathbb{C}\right\}. \nonumber
	\end{equation}
	Moreover, \(K \cap B = \{I_3\}\).
\end{lem}
\begin{proof}
	Let \(g \in K^{\mathbb{C}}\) and set \(x = \overline{g}^tg\). Since \(x\) is a positive-definite Hermitian matrix, there exists a Hermitian matrix \(X \in {\rm M}_3 \mathbb{C}\) (i.e. \(\overline{X} = X^t\)) such that \(x = \exp{(X)}\). Since \(\overline{x} = g^t\overline{g} = \sigma(g)^t\overline{\sigma(g)} = \Delta x^{-1} \Delta\), we obtain \(\Delta \overline{X} \Delta = -X\). Hence
	\begin{equation}
		X = \left(\begin{array}{ccc}
			a & b & 0 \\
			\overline{b} & 0 & -b \\
			0 & -\overline{b} & -a
		\end{array}\right),\ \ \ \ \ a \in \mathbb{R},\ b \in \mathbb{C}. \nonumber
	\end{equation}
	
	Set \(p = \exp{(X/2)}\) and \(k = gp^{-1}\). Then \(x = pp, \Delta\overline{p}\Delta = p^{-1}\) and
	\begin{align}
		\sigma(k) &= \Delta (g^{-1})^tp^t \Delta = \sigma(g)\Delta \exp{\left(\frac{X^t}{2}\right)}\Delta = g\Delta\exp{\left(\frac{\overline{X}}{2}\right)}\Delta = g\exp{\left(\frac{\Delta\overline{X}\Delta}{2}\right)} \nonumber\\
		&= g\exp{\left(-\frac{X}{2}\right)} = gp^{-1} = k, \nonumber\\
		\tau(k) &= \Delta\overline{k}\Delta = \Delta\overline{g}\overline{p^{-1}}\Delta = \Delta \left((g^{-1})^t\overline{x}\right)\overline{p^{-1}}\Delta 
		= \Delta (g^{-1})^t \Delta \Delta\overline{xp^{-1}}\Delta \nonumber\\ 
		&= \sigma(g)\Delta \overline{p}\Delta = gp^{-1} = k. \nonumber
	\end{align}
	Therefore, \(k \in K\), and hence \(g=kp\in KB\).
\end{proof}\vskip\baselineskip

Let \(I = \{\lambda \in \mathbb{C}\ | \ |\lambda| < 1\}\) denote the open unit disk. We define the following twisted loop groups.
\begin{align}
	&(\Lambda {\rm SL}_3 \mathbb{C})_{\sigma} = \{\gamma:S^1 \rightarrow {\rm SL}_3 \mathbb{C}\ | \ \text{\(\gamma\) is smooth in \(\lambda \in S^1\)},\ \sigma(\gamma(-\lambda)) = \gamma(\lambda)\}, \nonumber\\
	&(\Lambda^+_B {\rm SL}_3 \mathbb{C})_{\sigma} = \{\gamma \in (\Lambda {\rm SL}_3 \mathbb{C})_{\sigma}\ | \ \text{\(\gamma, \gamma^{-1}\) extend holomorphically to \(I\)},\ \gamma(0) \in B\}, \nonumber\\
	&(\Lambda G)_{\sigma} = \{\gamma \in (\Lambda {\rm SL}_3 \mathbb{C})_{\sigma}\ | \ \gamma(\lambda) \in G\}. \nonumber
\end{align}

The corresponding twisted loop algebra is defined by
\begin{equation}
	(\Lambda \mathfrak{sl}_3 \mathbb{C})_{\sigma} = \{A:S^1 \rightarrow \mathfrak{sl}_3 \mathbb{C}\ | \ \text{\(A\) is smooth in \(\lambda\)},\ \sigma(A(-\lambda)) = A(\lambda)\}. \nonumber
\end{equation}
Here and throughout, we use the same symbol \(\sigma\) to denote the induced involution on \(\mathfrak{sl}_3\mathbb{C}\).\vskip\baselineskip

We shall make use of the following Iwasawa decomposition.
\begin{theorem}[Balan, Dorfmeister \cite{BD2001}]
	There exists an open dense subset \(\mathcal{U} \subset (\Lambda {\rm SL}_3 \mathbb{C})_{\sigma}\) such that the multiplication
	\begin{equation}
		(\Lambda G)_{\sigma} \times (\Lambda^+_B {\rm SL}_3 \mathbb{C})_{\sigma} \rightarrow \mathcal{U} \nonumber
	\end{equation}
	is a real-analytic bijective diffeomorphism with respect to the natural smooth manifold.
\end{theorem}\vskip\baselineskip

We now briefly recall the DPW construction \cite{DPW1998}.\vskip\baselineskip

Let \(\Sigma\) be a simply-connected Riemann surface with local coordinate \(t\), and let \(\Omega^{1,0}_{\Sigma}\) denote the space of holomorphic \((1,0)\)-forms on \(\Sigma\). A DPW potential is a \((\Lambda\mathfrak{sl}_3\mathbb C)_\sigma\)-valued holomorphic \(1\)-form of the form
\begin{equation}
	\xi = \frac{1}{\lambda}C_{-1}dt + \sum_{j=0}^{\infty}C_j(t)\lambda^j \in \Omega_{\mathbb{C}}^{1,0} \otimes (\Lambda {\mathfrak{sl}_3 \mathbb{C}})_{\sigma} \nonumber
\end{equation}
satisfying
\begin{equation}
	S^{-1}\xi(t,\omega \lambda)S = \xi(t,\lambda),\ \ \ \ \ \omega = e^{i\frac{2}{3}\pi},\ S = \left(\begin{array}{ccc}
		1 & 0 & 0 \\
		0& \omega & 0 \\
		0& 0 & \omega^2 
	\end{array}\right). \nonumber
\end{equation}

The symmetry conditions imposed on \(\xi\) determine the form of the leading coefficient \(C_{-1}\).
\begin{lem}
	The matrix \(C_{-1}\) is of the form \(C_{-1} = \left(\begin{array}{ccc}
		0& c_2 & 0 \\
		0&  0 & c_2 \\
		c_1 &  0& 0 
	\end{array}\right)\) for some \(c_1, c_2 \in \mathbb{C}\).
\end{lem}
\begin{proof}
	Since \(\omega^2S^{-1}C_{-1}S = C_{-1}\), the matrix \(C_{-1}\) must be of the form
	\begin{equation}
		C_{-1} = \left(\begin{array}{ccc}
			0	& c_2 & 0 \\
			0	&  0& c_3 \\
			c_1 & 0 & 0
		\end{array}\right), \nonumber
	\end{equation}
	for some \(c_1,c_2,c_3 \in \mathbb{C}\). Furthermore, the condition \(\sigma(\xi(-\lambda)) = \xi(\lambda)\) implies \(\Delta C_{-1}^t \Delta = C_{-1}\) and hence \(c_2 = c_3\).
\end{proof}\vskip\baselineskip

Solve \(d\phi = \phi\xi\) with the initial condition \(\phi(0,\lambda) = I_3\). Then, on a sufficiently small neighborhood \(U \subset \mathbb{C}\) of \(t=0\), the solution \(\phi\) admits an Iwasawa factorization \(\phi = F\phi_+\), where
\(F \in (\Lambda G)_{\sigma}\), \(\phi_+ \in (\Lambda^+ {\rm SL}_3 \mathbb{C})_{\sigma}\). The following proposition describes the Maurer--Cartan form of the extended frame \(F\).

\begin{prop}
	There exists a real-valued function \(w:U\to\mathbb R\) such that \(\phi_+|_{\lambda = 0} = {\rm diag}(e^{\frac{w}{2}},1,e^{-\frac{w}{2}})\) and
	\footnotesize
	\begin{equation}
		F^{-1}dF = \left(\begin{array}{ccc}
			\frac{w_t}{2} & \lambda^{-1}c_2e^{\frac{w}{2}} & 0 \\
			0 & 0 & \lambda^{-1}c_2e^{\frac{w}{2}} \\
			\lambda^{-1}c_1e^{-w}  & 0 & -\frac{w_t}{2}
		\end{array}\right)dt + \left(\begin{array}{ccc}
			-\frac{w_{\overline{t}}}{2} & 0 & \lambda \overline{c}_1e^{-w} \\
			\lambda \overline{c}_2e^{\frac{w}{2}} & 0 & 0 \\
			0 & \lambda \overline{c}_2e^{-w}  & \frac{w_{\overline{t}}}{2}
		\end{array}\right)d\overline{t}. \nonumber
	\end{equation}
	\normalsize
\end{prop}
\begin{proof}
	Since \(S^{-1}\xi(t,\omega \lambda)S = \xi(t,\lambda)\) and \(\phi(0,\lambda) = I_3\), the uniqueness of solutions to \(d\phi = \phi\xi,\ \phi(0) = I_3\) implies that \(S^{-1}\phi(t,\omega \lambda)S = \phi(t,\lambda)\). By the uniqueness of the Iwasawa factorization, it follows that \(S^{-1} F(t,\overline{t},\omega\lambda)S = F(t,\overline{t},\lambda)\) and \(S^{-1}\phi_+(t,\overline{t},\omega\lambda)S = \phi_+(t,\overline{t},\lambda)\). Hence \(\phi_+|_{\lambda = 0}\) is diagonal. Writing \(\phi_+|_{\lambda = 0} = {\rm diag}(e^{\frac{w}{2}},1,e^{-\frac{w}{2}})\) and using \(F^{-1}dF = \phi_+\xi\phi_+^{-1} - d\phi_+ \phi_+^{-1}\) together with the reality condition \(\tau(F^{-1}dF) = F^{-1}dF\), we obtain the stated formula.
\end{proof}\vskip\baselineskip

Comparing the above expression with the frame equations in Proposition \ref{prop2.1}, we obtain the following corollary.
\begin{cor}
	Let \(f = TF\begin{pmatrix}
		0 \\ 1 \\ 0
	\end{pmatrix}\). Then \(f = f(t,\overline{t},\lambda):U \rightarrow \mathbb{R}^3\) defines a family of hyperbolic affine spheres parameterized by \(\lambda \in S^1\).
\end{cor}\vskip\baselineskip

The following example illustrates the DPW construction and recovers the affine sphere in Example \ref{ex1}.
\begin{example}
	Consider the DPW potential
	\begin{equation}
		\xi = \lambda^{-1}\left(\begin{array}{ccc}
			0& 1 & 0 \\
			0& 0 & 1 \\
			1 &  0& 0
		\end{array}\right)dt. \nonumber
	\end{equation}
	The solution of \(d\phi = \phi \xi,\ \phi(0) = I_3\) is given by
	\begin{equation}
		\phi = \exp{\left(\lambda^{-1}t\left(\begin{array}{ccc}
				0& 1 & 0 \\
				0& 0 & 1 \\
				1 & 0 & 0
			\end{array}\right)\right)}, \nonumber
	\end{equation}
	and \(\phi\) admits an Iwasawa factorization \(\phi = F\phi_+\) on \(\mathbb{C}\), where
	\begin{align}
		&F = \exp{\left(\lambda^{-1}t\left(\begin{array}{ccc}
				0& 1  & 0 \\
				0& 0 & 1 \\
				1 & 0 & 0
			\end{array}\right) + \lambda \overline{t}\left(\begin{array}{ccc}
				0& 0  & 1 \\
				1 &  0& 0 \\
				0& 1 & 0
			\end{array}\right)\right)}, \nonumber\\
		&\phi_+ = \exp{\left(-\lambda \overline{t}\left(\begin{array}{ccc}
				0& 0  & 1 \\
				1 &  0& 0 \\
				0& 1 & 0
			\end{array}\right)\right)}. \nonumber
	\end{align}
	Since
	\begin{align}
		TF &= T\exp{\left(\lambda^{-1}t\left(\begin{array}{ccc}
				0& 1  & 0 \\
				0& 0 & 1 \\
				1 & 0 & 0
			\end{array}\right) + \lambda \overline{t}\left(\begin{array}{ccc}
				0& 0  & 1 \\
				1 &  0& 0 \\
				0& 1 & 0
			\end{array}\right)\right)}T^{-1}T \nonumber\\
		&= \left(\begin{array}{ccc}
			e^{\lambda^{-1}t + \lambda \overline{t}} & 0 & 0 \\
			0& e^{\lambda^{-1}t\omega + \lambda \overline{t}\omega^2} & 0 \\
			0& 0 & e^{\lambda^{-1}t\omega^2 + \lambda \overline{t}\omega}
		\end{array}\right)T, \nonumber
	\end{align}
	\(f = f(t,\overline{t},\lambda):\mathbb{C} \rightarrow \mathbb{R}^3\) is given by
	\begin{equation}
		f = TF\begin{pmatrix}
			0 \\ 1 \\ 0
		\end{pmatrix} = \frac{1}{\sqrt{3}}\begin{pmatrix}
			e^{\lambda^{-1}t + \lambda \overline{t}} \\
			e^{\lambda^{-1}t\omega + \lambda \overline{t}\omega^2} \\
			e^{\lambda^{-1}t\omega^2 + \lambda \overline{t}\omega}
		\end{pmatrix}. \nonumber
	\end{equation}
	
	Writing \(f = \begin{pmatrix}
		x_1 \\ x_2 \\ x_3
	\end{pmatrix}\), we obtain \(x_1x_2x_3 = \frac{\sqrt{3}}{9}\). Hence, we recover the hyperbolic affine sphere of Example \ref{ex1}. \qed
\end{example}

\section{Iwasawa factorization of \(\phi\)}
In this paper, we investigate translationally equivariant affine spheres arising from translationally invariant solutions of the Tzitz\'eica equation. To construct such surfaces, we introduce a Delaunay-type DPW potential.
\small
\begin{align}
	\xi &= \lambda^{-1}i\left(\begin{array}{ccc}
		0 & a^{-1} & 0 \\
		0 & 0  & a^{-1} \\
		a^2 & 0  & 0 
	\end{array}\right)dt + m\left(\begin{array}{ccc}
		1 & 0   & 0  \\
		0 & 0 & 0   \\
		0 & 0  & -1
	\end{array}\right)dt + \lambda i\left(\begin{array}{ccc}
		0 &  0  & a^2 \\
		a^{-1} & 0  & 0   \\
		0 	& a^{-1} & 0 
	\end{array}\right)dt \nonumber\\
	&= \left(\begin{array}{ccc}
		m & \lambda^{-1}ia^{-1} & \lambda ia^2 \\
		\lambda ia^{-1} & 0 & \lambda^{-1}ia^{-1} \\
		\lambda^{-1}ia^2 & \lambda ia^{-1} & -m
	\end{array}\right)dt =: Xdt, \nonumber
\end{align}
\normalsize
where \(a > 0,\ m \in \mathbb{R}\). In this section, we characterize the corresponding Iwasawa factorization by a certain linear differential equation and prove its existence on \(\mathbb{C}\) except away from countably many lines.

\subsection{The Delaunay potential and the Weierstrass \(\wp\)-function}
First, we show that the Delaunay-type potential gives rise to a translationally invariant solution of the Tzitz\'eica equation and that the solution can be expressed in terms of the Weierstrass elliptic function \(\wp\).\vskip\baselineskip

The solution of \(d\phi = \phi \xi,\ \phi(0) = I_3\) is given by
\begin{align}
	\phi(t,\lambda) &= \exp{(tX)} = \exp{\left(t\left(\begin{array}{ccc}
			m & \lambda^{-1}ia^{-1} & \lambda ia^2 \\
			\lambda ia^{-1} & 0 & \lambda^{-1}ia^{-1} \\
			\lambda^{-1}ia^2 & \lambda ia^{-1} & -m
		\end{array}\right)\right)}. \nonumber
\end{align}
We split \(\phi = F\phi_+\) by the Iwasawa factorization, where \(F \in (\Lambda G)_{\sigma}, \phi_+ \in (\Lambda^+_B {\rm SL}_3 \mathbb{C})_{\sigma}\).

\begin{lem}
	We have \(\phi_+ = \phi_+(x,\lambda)\), where \(t = x + iy\). Moreover, \(f = TF\begin{pmatrix}
		0 \\ 1 \\ 0
	\end{pmatrix}\) defines a translationally equivariant hyperbolic affine sphere satisfying
	\begin{equation}
		f(t + iy') = T\phi(iy')T^{-1}f(t),\ \ \ \ \ T\phi(iy)T^{-1} \in {\rm SL}_3 \mathbb{R} \nonumber
	\end{equation}
	for all \(\lambda \in S^1\).
\end{lem}
\begin{proof}
	Since \(\phi(t,\lambda) = \phi(iy,\lambda)\phi(x,\lambda)\) and \(\phi(iy,\lambda) \in (\Lambda G)_{\sigma}\), the positive part of \(\phi(t,\lambda)\) coincides with that of \(\phi(x,\lambda)\). Hence, \(\phi_+(t,\overline{t},\lambda) = \phi_+(x,\lambda)\). The latter follows from \(\overline{\phi(iy)} = \Delta\phi(iy)\Delta\) and \(\overline{T}\Delta T^{-1} = I_3\).
\end{proof}\vskip\baselineskip

Set
\begin{equation}
	\phi_0 = \phi_+(t,\overline{t},\lambda)|_{\lambda = 0} = \left(\begin{array}{ccc}
		e^{\frac{u}{2}} & 0 & 0  \\
		0 & 1 & 0  \\
		0 & 0  & e^{-\frac{u}{2}}
	\end{array}\right). \nonumber
\end{equation}
The following proposition then holds.

\begin{prop}
	\(u(t,\overline{t}) = u(x)\) is a translationally invariant solution to the Tzitz\'eica equation
	\begin{equation}
		u_{t\overline{t}} = a^{-2}e^u - a^4e^{-2u}, \nonumber
	\end{equation}
	with the initial condition \(u(0) = 0, u_t(0) = m\). Equivalently, \(u(x)\) satisfies the radial Tzitz\'eica equation
	\begin{equation}
		u_{xx} = 4a^{-2}e^u - 4a^4e^{-2u}, \nonumber
	\end{equation}
	with the initial condition \(u(0) = 0, u_x(0) = 2m\).
\end{prop}
\begin{proof}
	Let \(\alpha = F^{-1}dF\). Since \(d\alpha + \alpha \wedge \alpha = 0\), it follows that \(u\) is a solution of the Tzitz\'eica equation. From \(\phi_+(0,\lambda) = I\), we obtain \(u(0) = 0\). Let \(\phi_+ = \phi_0 + \lambda \phi_1 + O(\lambda^2)\), and write \(\xi = \lambda^{-1}iCdt + Ndt + \lambda iC^tdt\). Then, using
	\begin{equation}
		\alpha = F^{-1}dF = \phi_+\xi \phi_+^{-1} -d\phi_+ \phi_+^{-1}, \nonumber
	\end{equation}
	we obtain
	\begin{equation}
		(\phi_0)_t = \phi_0N + i\phi_1C - (\phi_0)_t - i\phi_0C\phi_0^{-1}\phi_1. \nonumber
	\end{equation}
	
	Setting \(t = 0\), and using the fact that \(\phi_1|_{t=0} = 0\), we obtain
	\small
	\begin{align}
		&\left(\begin{array}{ccc}
			\frac{u_t(0)}{2}e^{\frac{u(0)}{2}} & 0  & 0  \\
			0 & 0 & 0  \\
			0 & 0  & -\frac{u_t(0)}{2}e^{-\frac{u(0)}{2}}
		\end{array}\right) \nonumber\\ 
		&= \left(\begin{array}{ccc}
			e^{\frac{u(0)}{2}} & 0  & 0  \\
			0 & 1 & 0  \\
			0 & 0  & e^{-\frac{u(0)}{2}}
		\end{array}\right)\left(\begin{array}{ccc}
			m & 0  & 0  \\
			0 & 0 & 0  \\
			0 & 0  & -m
		\end{array}\right) - \left(\begin{array}{ccc}
			\frac{u_t(0)}{2}e^{\frac{u(0)}{2}} & 0  & 0  \\
			0 & 0 & 0  \\
			0 & 0  & -\frac{u_t(0)}{2}e^{-\frac{u(0)}{2}}
		\end{array}\right). \nonumber
	\end{align}
	\normalsize
	Hence, we conclude that \(u_t(0) = m\), and therefore \(u_x(0)=2m\).
\end{proof}\vskip\baselineskip

By integrating the Tzitz\'eica equation once, we obtain that \(u\) satisfies
\begin{equation}
	\frac{1}{2}(u_x)^2 = 4a^{-2}e^u + 2a^4e^{-2u} + 6M, \nonumber
\end{equation}
where
\begin{equation}
	M = \frac{1}{3}(m^2 - 2a^{-2} - a^4). \nonumber
\end{equation}

\begin{prop}\label{prop3.2}
	The function \(2a^{-2}e^u + M\) can be expressed in terms of the Weierstrass \(\wp\)-function as \(\wp(x+\beta) = 2a^{-2}e^u + M\), where
	\begin{align}
		&\wp(\beta) = 2a^{-2} + M = \frac{1}{3}(m^2 + 4a^{-2} - a^4), \nonumber\\
		&\wp_x(\beta) =  4a^{-2}m. \nonumber
	\end{align}
\end{prop}
\begin{proof}
	Set \(p(x) = 2a^{-2}e^u + M\). Then
	\begin{align}
		(p_x)^2 &= 4a^{-4}(u_x)^2e^{2u} = 4p^3 - 12M^2p + 8(M^3+2). \nonumber
	\end{align}
	Hence, \(p(x) = \wp(x+\beta)\) satisfies the differential equation of  the Weierstrass \(\wp\)-function. Therefore, there exists a constant \(\beta\) such that \(p(x) = \wp(x+\beta)\). Moreover,
	\begin{align}
		&\wp(\beta) = 2a^{-2} + M = \frac{1}{3}(m^2 + 4a^{-2} - a^4),\ \ \ \ \ \wp_x(\beta) =  4a^{-2}m. \nonumber
	\end{align}
\end{proof}\vskip\baselineskip

Let \(\omega_1\) and \(\omega_2\) be the fundamental periods of the Weierstrass elliptic function \(\wp\), and define
\begin{equation}
	\tilde{\Lambda} = \{x \in \mathbb{C}\ | \ x + \beta \in \omega_1\mathbb{Z}+\omega_2\mathbb{Z} \ \text{or}\ \wp(x+\beta) = M\}.\nonumber
\end{equation}

The translationally invariant solution \(u = u(x)\) is defined on the connected component of \(\mathbb{R} \backslash \tilde{\Lambda}\) containing \(0\).

\subsection{Iwasawa factorization away from countably many lines}
In this section, we show that the Iwasawa factorization induces a certain differential equation and we describe the Iwasawa factorization on \(\mathbb{C}\) by using the solutions.\vskip\baselineskip

The existence of the Iwasawa factorization is equivalent to the solvability of a certain linear differential equation.
\begin{prop}\label{prop3.3}
	Let \(\phi_+ = \phi_+(x,\lambda)\) be a solution to
	\begin{equation}\label{L}
		(\phi_+)_x = \left(\begin{array}{ccc}
			\frac{u_x}{2} & 0 & 2\lambda i a^2e^{-u} \\
			2\lambda ia^{-1}e^{\frac{u}{2}} & 0 & 0 \\
			0 & 2\lambda i a^{-1}e^{\frac{u}{2}} & -\frac{u_x}{2}
		\end{array}\right)\phi_+,\ \ \ \ \ \phi_+(0,\lambda) = I_3.
	\end{equation}
	Then \(\phi_+ = \phi_+(x) \in (\Lambda^+_B{\rm SL}_3 \mathbb{C})_{\sigma}\) and \(F = \phi \phi_+^{-1}\) satisfies \(F^{-1}dF = \alpha\).
\end{prop}
\begin{proof}
	Set \(L = (\phi_+)_x\phi_+^{-1}\). By \(\phi_+ = \phi_+(x,\lambda)\),
	\begin{equation}
		\alpha = F^{-1}dF = \phi_+\xi \phi_+^{-1} - d\phi_+\phi_+^{-1} \nonumber
	\end{equation}
	is equivalent to
	\begin{align}\label{A}
		(\phi_+)_x &= \phi_+X - \left(\begin{array}{ccc}
			0 & \lambda^{-1}ia^{-1}e^{\frac{u}{2}} & -\lambda i a^2e^{-u} \\
			-\lambda ia^{-1}e^{\frac{u}{2}} & 0 & \lambda^{-1}i a^{-1}e^{\frac{u}{2}} \\
			\lambda^{-1}ia^2e^{-u} & -\lambda i a^{-1}e^{\frac{u}{2}} & 0
		\end{array}\right)\phi_+ \nonumber\\
		&=: \phi_+X - L_1\phi_+, \nonumber\\
		0 &= \phi_+X - \left(\begin{array}{ccc}
			\frac{u_x}{2} & \lambda^{-1}ia^{-1}e^{\frac{u}{2}} & \lambda i a^2e^{-u} \\
			\lambda ia^{-1}e^{\frac{u}{2}} & 0 & \lambda^{-1}i a^{-1}e^{\frac{u}{2}} \\
			\lambda^{-1}ia^2e^{-u} & \lambda i a^{-1}e^{\frac{u}{2}} & -\frac{u_x}{2}
		\end{array}\right)\phi_+ \nonumber\\
		&=: \phi_+X - L_2\phi_+.
	\end{align}
	Hence, it suffices to show that \(\phi_+\) satisfies (\ref{A}).\vskip\baselineskip
	
	We have
	\begin{align}
		&(L_2\phi_+)_x \nonumber\\
		&= \left(\begin{array}{ccc}
			\frac{u_{xx}}{2} & \frac{\lambda^{-1}}{2}ia^{-1}u_xe^{\frac{u}{2}} & -\lambda i a^2u_xe^{-u} \\
			\frac{\lambda}{2} ia^{-1}u_xe^{\frac{u}{2}} & 0 & \frac{\lambda^{-1}}{2}i a^{-1}u_xe^{\frac{u}{2}} \\
			-\lambda^{-1}ia^2u_xe^{-u} & \frac{\lambda}{2} i a^{-1}u_xe^{\frac{u}{2}} & -\frac{u_{xx}}{2}
		\end{array}\right)\phi_+ + L_2L\phi_+ \nonumber\\
		&= \text{\scriptsize \(\left(\begin{array}{ccc}
				\frac{u_{xx}}{2} + \frac{(u_x)^2}{4} - 2a^{-2}e^u & \frac{\lambda^{-1}}{2}ia^{-1}u_xe^{\frac{u}{2}} - 2\lambda^2ae^{-\frac{u}{2}} & -\frac{\lambda}{2} i a^2u_xe^{-u} \\
				\lambda i a^{-1}u_xe^{\frac{u}{2}} & -2a^{-2}e^u & -2\lambda^2ae^{-\frac{u}{2}} \\
				-\frac{\lambda^{-1}}{2}ia^2u_xe^{-u} -2\lambda^2a^{-2}e^u & -\frac{\lambda}{2} ia^{-1}u_xe^{\frac{u}{2}} & -\frac{u_{xx}}{2} + \frac{(u_x)^2}{4} - 2a^4e^{-2u}
			\end{array}\right)\phi_+\)} \nonumber\\
		&= \text{\footnotesize\(\left(\begin{array}{ccc}
				\frac{u_x}{2} & 0 & 2\lambda i a^2e^{-u} \\
				2\lambda ia^{-1}e^{\frac{u}{2}} & 0 & 0 \\
				0 & 2\lambda i a^{-1}e^{\frac{u}{2}} & -\frac{u_x}{2}
			\end{array}\right)\left(\begin{array}{ccc}
				\frac{u_x}{2} & \lambda^{-1}ia^{-1}e^{\frac{u}{2}} & \lambda i a^2e^{-u} \\
				\lambda ia^{-1}e^{\frac{u}{2}} & 0 & \lambda^{-1}i a^{-1}e^{\frac{u}{2}} \\
				\lambda^{-1}ia^2e^{-u} & \lambda i a^{-1}e^{\frac{u}{2}} & -\frac{u_x}{2}
			\end{array}\right)\phi_+\)} \nonumber\\
		&= LL_2\phi_+. \nonumber
	\end{align}
	and hence \(L_2\phi_+ = \phi_+\gamma\) for some constant matrix \(\gamma \in (\Lambda {\rm SL}_3 \mathbb{C})_{\sigma}\). Since \(\phi_+(0,\lambda) = I_3\), it follows that \(\gamma = L_2(0,\lambda) = X\). Therefore, we obtain the second equation in (\ref{A}). Using this equation together with (\ref{L}) and the relation \(L = L_2 - L_1\), we obtain
	\begin{equation}
		(\phi_+)_x = L\phi_+ = (L_2 - L_1)\phi_+ = \phi_+X - L_1 \phi_+. \nonumber
	\end{equation}
	Hence, the first equation in (\ref{A}) follows, and moreover \(F^{-1}dF=\alpha\). It follows from \((\phi_+)_x\phi_+^{-1}|_{\lambda = 0} = {\rm diag}(u_x/2,0,-u_x/2)\) that \(\phi_+ \in (\Lambda^+_B{\rm SL}_3 \mathbb{C})_{\sigma}\).
\end{proof}\vskip\baselineskip

Hence, we obtain the Iwasawa factorization \(\phi = F\phi_+\) on the connected component of \(D\) containing \(0\). Since the Weierstrass elliptic function \(\wp\) is meromorphic on \(\mathbb{C}\), the same argument applies on all of \(\mathbb{C} \backslash \tilde{\Lambda}\).\vskip\baselineskip

Let \(\tilde{\Phi}_+^{-1} = \tilde{\Phi}_+^{-1}(t,\lambda)\) be a solution of
\begin{equation}\label{A'}
	(\tilde{\Phi}_+^{-1})_t =
	\tilde{\Phi}_+^{-1}\left(\begin{array}{ccc}
		-\frac{\wp_t(t+\beta)}{\wp(t+\beta) - M} & 0 & -2\lambda ia^2 \\
		-2\lambda ia^{-1} & 0 & 0 \\
		0 & -2\lambda ia^{-1} & \frac{\wp_t(t+\beta)}{\wp(t+\beta) - M}
	\end{array}\right),\ \ \ \ \ \tilde{\Phi}_+^{-1}(0,\lambda) = I_3, \tag{A}
\end{equation}
on a simply-connected subset of \(\mathbb{C} \backslash \tilde{\Lambda}\), where
\begin{equation}
	\wp(\beta) = \frac{1}{3}(m^2 + 4a^{-2} - a^4),\ \ \ \ \ \wp_t(\beta) =  4a^{-2}m, \nonumber
\end{equation}
and set
\begin{equation}
	\tilde{F}(t,\lambda) = \phi(t,\lambda)\tilde{\Phi}_+(t,\lambda)^{-1} \nonumber
\end{equation}
and
\begin{equation}
	g(t) = \left(\begin{array}{ccc}
		\frac{a^2(\wp(t+\beta) - M)}{2} & 0&0 \\
		0 & 1 & 0 \\
		0 & 0 & \frac{2}{a^2(\wp(t+\beta) - M)}
	\end{array}\right). \nonumber
\end{equation}

\begin{lem}\label{lem3.2}
	We have \(\Delta\overline{\tilde{F}({\bar t},\lambda)}\Delta = \tilde{F}(t,\lambda)g(t)\) for all \(\lambda \in S^1\).
\end{lem}
\begin{proof}
	First, we prove
	\begin{align}\label{E}
		\tilde{\Phi}_+^{-1}\left(\begin{array}{ccc}
			\frac{\wp_t(t+\beta)}{2(\wp(t+\beta) - M)} & \lambda^{-1}\frac{ia(\wp(t+\beta) - M) }{2}& \lambda i a^2 \\
			\lambda i a^{-1} & 0 & \lambda^{-1}\frac{ia(\wp(t+\beta) - M)}{2} \\
			\frac{4\lambda^{-1}ia^{-2}}{(\wp(t+\beta) - M)^2} & \lambda i a^{-1} & -\frac{\wp_t(t+\beta)}{2(\wp(t+\beta) - M)}
		\end{array}\right) 
		= X\tilde{\Phi}_+^{-1}.
	\end{align}
	\normalsize
	Let \(V = {\rm diag}(e^{-u/2},1,e^{u/2})\). By Proposition \ref{prop3.2} and the fact that \((\tilde{\Phi}_+^{-1}(x,\lambda))_x = (\tilde{\Phi}_+^{-1}(t,\lambda))_t|_{t=x}\), we have
	\begin{equation}
		(\tilde{\Phi}_+^{-1}(x,\lambda)V^{-1}(x)\phi_+(x,\lambda))_x = 0 \nonumber
	\end{equation}
	for all \(x \in \mathbb{C} \backslash \tilde{\Lambda}\). Since \(\tilde{\Phi}_+^{-1}(0,\lambda) = I_3, \phi_+(0,\lambda) = I_3\) and \(u(0) = 0\), it follows that \(\tilde{\Phi}_+(x,\lambda)^{-1} = \phi_+(x,\lambda)^{-1}V(x)\) for all \(x \in \mathbb{R} \backslash \tilde{\Lambda}\). Consequently, by (\ref{A}), \(\tilde{\Phi}_+^{-1}\) satisfies (\ref{E}) on \(x \in \mathbb{R} \backslash \tilde{\Lambda}\). Since both sides of (\ref{E}) are meromorphic on \(\mathbb{C}\), the identity theorem for holomorphic functions implies that (\ref{E}) holds on \(\mathbb{C} \backslash \tilde{\Lambda}\).\vskip\baselineskip
	
	From (\ref{E}), we obtain
	\begin{align}
		\tilde{F}_t &= \phi X
		\tilde{\Phi}_+^{-1} + \phi(\tilde{\Phi}_+^{-1})_t 
		= \tilde{F}\left(\begin{array}{ccc}
			-\frac{\wp_t}{2(\wp - M)} & \lambda^{-1} \frac{ia(\wp-M)}{2} & -\lambda ia^2 \\
			-\lambda ia^{-1} & 0 & \lambda^{-1} \frac{ia(\wp-M)}{2} \\
			\frac{4\lambda^{-1}ia^{-2}}{(\wp - M)^2} & -\lambda ia^{-1} & \frac{\wp_t}{2(\wp - M)}
		\end{array}\right). \nonumber
	\end{align}
	Since \(\overline{\wp({\bar t})} = \wp(t)\), we have
	\begin{align}
		&\Delta\overline{(\tilde{F}^{-1}\tilde{F}_t)({\bar t},\lambda)}\Delta = \left(\tilde{F}(t,\lambda)g(t)\right)^{-1}\left(\tilde{F}(t,\lambda)g(t)\right)_t. \nonumber
	\end{align}
	\normalsize
	Hence, from \(\tilde{F}(0,\lambda)g(0) = I_3\) we obtain \(\Delta\overline{\tilde{F}({\bar t},\lambda)}\Delta = \tilde{F}(t,\lambda)g(t)\).
\end{proof}\vskip\baselineskip

Hence, we obtain the Iwasawa factorization of \(\phi\) on \(\mathbb{C}\) away from countably many lines.
\begin{theorem}\label{thm3.4}
	Let \(M = \frac{1}{3}(m^2 - a^4 - 2a^{-2})\) and \(\tilde{\Phi}_+^{-1}\) a solution of (\ref{A'}) and define \(\mathcal{E}(x) = {\rm diag}(\varepsilon(x),1,\varepsilon(x))\), where
	\begin{equation}
		\varepsilon(x) = \left\{
		\begin{array}{cc}
			1 & {\rm if}\ \wp(x+\beta) - M > 0,\\
			-1 & {\rm if}\ \wp(x+\beta) - M < 0
		\end{array},\ \ \ \ \ x \in \mathbb{R}.\nonumber
		\right.
	\end{equation}
	Then \(\phi(t,\lambda)\) admits an Iwasawa factorization \(\phi = F\phi_+\) on a simply-connected subspace of
	\begin{equation}
		\mathbb{C} \backslash \{t = x+iy \in \mathbb{C}\ | \ x \in \tilde{\Lambda},\ y \in \mathbb{R}\}, \nonumber
	\end{equation}
	where \(F = F(t,\overline{t},\lambda) \in (\Lambda G)_{\sigma}\) and \(\phi_+ = \phi_+(t,\overline{t},\lambda) \in (\Lambda^+_B {\rm SL}_3 \mathbb{C})_{\sigma}\) are given by
	\begin{align}
		&F(t,{\bar t},\lambda) = \phi(iy,\lambda)\tilde{F}(x,\lambda)\mathcal{E}(x)g(x)^{\frac{1}{2}} \in (\Lambda G)_{\sigma}
		, \nonumber\\ 
		&\phi_+(t,{\bar t},\lambda) = g(x)^{-\frac{1}{2}}\mathcal{E}(x)\tilde{\Phi}_+(x,\lambda) 
		\in (\Lambda^+_B {\rm SL}_3 \mathbb{C})_{\sigma}
		. \nonumber
	\end{align}
	\normalsize
\end{theorem}
\begin{proof}
	For \(x \in \mathbb{R} \backslash \tilde{\Lambda}\), we have \(\varepsilon(x)(\wp(x+\beta)-M) > 0\), and hence \(g(x)^{\frac{1}{2}} \in {\rm SL}_3\mathbb{R}\). Since \(\phi(iy) \in (\Lambda G)_{\sigma}\), Lemma \ref{lem3.2} implies that
	\begin{align}
		\tau(F(t,{\bar t},\lambda)) &= \tau(\phi(iy))\Delta\overline{\tilde{F}(x,\lambda)}\mathcal{E}(x)g(x)^{\frac{1}{2}}\Delta = \phi(iy)\tilde{F}(x,\lambda)g(x)\Delta\mathcal{E}(x)\Delta g(x)^{-\frac{1}{2}} \nonumber\\
		&= F(t,{\bar t},\lambda) \nonumber
	\end{align}
	for all \(\lambda \in S^1\). Moreover, from the twisted condition satisfied by the right-hand side of (\ref{A'}) and the condition \(\tilde{\Phi}_+^{-1}(0,\lambda) = I_3\), we obtain \(\sigma(\tilde{\Phi}_+^{-1}(x,-\lambda)) = \tilde{\Phi}_+^{-1}(x,\lambda)\). Furthermore, by solving (\ref{A'}) for \(\lambda = 0\), we obtain \(\tilde{\Phi}_+^{-1}(x,0) = g(x)^{-1}\). Hence, \(F \in (\Lambda G)_{\sigma}\) and \(\phi_+ \in (\Lambda^+_B {\rm SL}_3 \mathbb{C})_{\sigma}\).
\end{proof}\vskip\baselineskip

Next, we express \(F\) in terms of the Weierstrass functions \(\wp,\zeta\) and \(\sigma\),  and construct translationally equivariant hyperbolic affine surfaces via \(f = TF\begin{pmatrix}
	0 \\ 1 \\ 0
\end{pmatrix}\).

\section{Translationally equivariant affine spheres}
In this section, we solve the scalar equation corresponding to (\ref{A'}) in terms of the Weierstrass functions \(\wp, \zeta\), and \(\sigma\). We then fix \(\lambda=-i\) and construct the corresponding affine spheres using these Weierstrass functions. Finally, we classify translationally equivariant hyperbolic affine spheres in terms of the parameters \(a\) and \(m\) appearing in the Delaunay-type potential \(\xi\).

\subsection{The scalar equation}\label{sec4.1}
We now derive the scalar equation corresponding to (\ref{A'}). Suppose that \((Y^{(1)},Y^{(2)},Y^{(3)})\) satisfies
\begin{equation}
	\partial_t(Y^{(1)},Y^{(2)},Y^{(3)})
	= (Y^{(1)},Y^{(2)},Y^{(3)})\left(\begin{array}{ccc}
		-\frac{\wp_t}{\wp-M} & 0 & -2\lambda ia^2 \\
		-2\lambda ia^{-1} & 0 & 0 \\
		0 & -2\lambda ia^{-1} & \frac{\wp_t}{\wp-M}
	\end{array}\right). \nonumber
\end{equation}
Then
\begin{equation}
	\left\{
	\begin{array}{l}
		(Y^{(1)})_t = -\frac{\wp_t(t+\beta)}{\wp(t+\beta)-M}Y^{(1)} - 2\lambda ia^{-1}Y^{(2)},\\
		(Y^{(2)})_t = -2\lambda ia^{-1}Y^{(3)},\\
		(Y^{(3)})_t = \frac{\wp_t(t+\beta)}{\wp(t+\beta)-M}Y^{(3)} - 2\lambda ia^2Y^{(1)},
	\end{array}
	\right. \nonumber
\end{equation}
and
\begin{align}\label{B}
	&Y^{(3)} = \frac{ia}{2\lambda}(Y^{(2)})_t,\ \ \ \ \ Y^{(1)} = -\frac{1}{4\lambda^2a}\left\{(Y^{(2)})_{tt} - \frac{\wp_t(t+\beta)}{\wp(t+\beta) - M}(Y^{(2)})_t \right\}, \nonumber\\
	&(Y^{(2)})_{ttt} - 6\{\wp(t+\beta) + M\}(Y^{(2)})_t = 8\lambda^3iY^{(2)}. 
\end{align}\vskip\baselineskip

We seek a solution to (\ref{B}) of the form
\begin{equation}
	Y = \frac{\sigma(t + \beta - \alpha)}{\sigma(t + \beta)}e^{(\zeta(\alpha) + k)t}, \nonumber
\end{equation}
where \(\zeta, \sigma\) are the Weierstrass functions. We have
\begin{align}
	Y_t &= \{\zeta(t+\beta-\alpha) - \zeta(t+\beta) + \zeta(\alpha) + k\}Y \nonumber\\
	&= \left\{\frac{1}{2}\frac{\wp_t(t+\beta) + \wp_t(\alpha)}{\wp(t+\beta) - \wp(\alpha)} + k\right\}Y \nonumber\\
	&=: \tilde{\psi}Y, \nonumber\\
	Y_{tt} &= (\tilde{\psi}_t + \tilde{\psi}^2)Y. \nonumber
\end{align}

\begin{lem}\label{lem4.1}
	Set
	\begin{equation}
		\psi = \frac{1}{2}\frac{\wp_t(t+\beta) + \wp_t(\alpha)}{\wp(t+\beta) - \wp(\alpha)}. \nonumber
	\end{equation}
	Then
	\begin{equation}
		\psi_t + \psi^2 = 2\wp(t+\beta) + \wp(\alpha). \nonumber
	\end{equation}
\end{lem}
\begin{proof}
	From \((\wp_t)^2 = 4\wp^3 - 12M^2\wp + 8(M^3+2)\), it follows that
	\begin{equation}
		\wp_t(\wp_{tt} -6\wp^2 + 6M^2) = 0. \nonumber
	\end{equation}
	We then compute
	\begin{align}
		&4(\wp - \wp(\alpha))^2(\psi_t + \psi^2) 
		= 2\wp_{tt}(\wp - \wp(\alpha)) - 2(\wp_t + \wp_t(\alpha))\wp_t + (\wp_t + \wp_t(\alpha))^2 \nonumber\\
		&= 2(6\wp^2 - 6M^2)(\wp - \wp(\alpha)) -(4\wp^3 - 12M^2\wp + 8(M^3+2)) + (\wp_t(\alpha))^2 \nonumber\\
		&= 8\wp^3 -12\wp(\alpha)\wp^2 + 4(\wp(\alpha))^3 \nonumber\\
		&= 4(\wp - \wp(\alpha))^2(2\wp + \wp(\alpha)). \nonumber
	\end{align}
	Consequently, we obtain \(\psi_t + \psi^2 = 2\wp + \wp(\alpha)\).
\end{proof}\vskip\baselineskip

\begin{prop}
	We have
	\begin{equation}
		Y_{ttt} - \{6\wp - 3\wp(\alpha) + 3k^2\}Y_t = \{- 2k^3 + 6\wp(\alpha)k - 2\wp_t(\alpha)\}Y. \nonumber
	\end{equation}
	Moreover, if \(\alpha, k\) satisfy
	\begin{align}\label{C}
		&-3\wp(\alpha) + 3k^2 = 6M, \nonumber\\
		&-2k^3 + 6\wp(\alpha)k - 2\wp_t(\alpha) = 8\lambda^3i 
	\end{align}
	then \(Y\) is a solution to (\ref{B}).
\end{prop}
\begin{proof}
	We have
	\begin{equation}
		Y_{ttt} = (\tilde{\psi}_{tt} + 3\tilde{\psi}\tilde{\psi}_t + \tilde{\psi}^3)Y. \nonumber
	\end{equation}
	From Lemma \ref{lem4.1}, it follows that
	\begin{align}
		\psi_t &= -\psi^2 + 2\wp + \wp(\alpha), \nonumber\\
		\psi_{tt} &= -2\psi\psi_t + 2\wp_t = 2\psi^3 - 2\psi(2\wp + \wp(\alpha)) + 4\psi(\wp - \wp(\alpha)) - 2\wp_t(\alpha) \nonumber\\
		&= 2\psi^3 - 6\wp(\alpha)\psi - 2\wp_t(\alpha). \nonumber
	\end{align}
	A straightforward computation using the above identities yields
	\begin{align}
		&\tilde{\psi}_{tt} + 3\tilde{\psi}\tilde{\psi}_t + \tilde{\psi}^3 
		= \psi_{tt} + 3\psi\psi_t + 3k\psi_t + (\psi + k)^3\nonumber\\
		&= \text{\footnotesize\(2\psi^3 - 6\wp(\alpha)\psi - 2\wp_t(\alpha) - 3\psi^3 + 3\psi(2\wp + \wp(\alpha)) - 3k\psi^2 + 3k(2\wp + \wp(\alpha)) + (\psi + k)^3\)} \nonumber\\
		&= (6\wp - 3\wp(\alpha) + 3k^2)\psi + 6k\wp - 2\wp_t(\alpha) + 3\wp(\alpha)k + k^3 \nonumber\\
		&= (6\wp - 3\wp(\alpha) + 3k^2)\tilde{\psi} - 2k^3 + 6\wp(\alpha)k - 2\wp_t(\alpha). \nonumber
	\end{align}
	Finally, from \(Y_t = \tilde{\psi}Y\), the results stated above follow.
\end{proof}\vskip\baselineskip

\begin{lem}
	If \(k, \alpha\) satisfy (\ref{C}), then
	\begin{equation}
		\wp(\alpha) = k^2 - 2M,\ \ \ \ \ \wp_t(\alpha) =  2k^3 - 6Mk - 4\lambda^3 i \nonumber
	\end{equation}
	and
	\begin{equation}\label{D}
		k^3 - 3Mk = i(\lambda^3 + \lambda^{-3}).
	\end{equation}
\end{lem}
\begin{proof}
	From \((\wp_t(\alpha))^2 = 4\wp(\alpha)^3 - 12M^2\wp(\alpha) + 8(M^3+2)\), we obtain
	\begin{align}
		0 &= (\wp_t(\alpha))^2 - 4\wp(\alpha)^3 + 12M^2\wp(\alpha) - 8(M^3+2) \nonumber\\
		&= (2k^3 - 6Mk - 4\lambda^3 i)^2 - 4(k^2 - 2M)^3 + 12M^2(k^2 - 2M) - 8(M^3+2) \nonumber\\
		&= -16\lambda^3 i(k^3 - 3Mk - i\lambda^3 -i\lambda^{-3}), \nonumber
	\end{align}
	and hence, \(k^3 - 3Mk = i(\lambda^3 + \lambda^{-3})\).
\end{proof}\vskip\baselineskip

Fix \(\lambda\) and solve (\ref{D}). Then the fundamental solutions of (\ref{B}) can be expressed in terms of Weierstrass functions. Consequently, so can the corresponding affine sphere.

\subsection{Hyperbolic affine spheres}
For \(t = x+iy \in \mathbb{C}\), a hyperbolic affine sphere \(f\) is given by
\begin{align}
	f(t,{\bar t},\lambda) &= TF\begin{pmatrix}
		0 \\ 1 \\ 0
	\end{pmatrix} = T\phi(t,\lambda) \tilde{\Phi}_+(x,\lambda)^{-1}\begin{pmatrix}
		0 \\ 1 \\ 0
	\end{pmatrix}. \nonumber
\end{align}

From now on, we fix \(\lambda = -i\) and \(k_1 = -k_3 = \sqrt{3M} =:k, k_2 = 0\).

\subsubsection{The case \(M \neq 0\)}
Assume that \(M = \frac{1}{3}(m^2 - a^4 - 2a^{-2}) \neq 0\).

\begin{prop}\label{prop4.2}
	Let \(\alpha_1,\alpha_2,\alpha_3\) satisfy
	\begin{align}
		&\wp(\beta) = 2a^{-2} + M = \frac{1}{3}(m^2 + 4a^{-2} - a^4),\ \ \ \ \ \wp_t(\beta) =  4a^{-2}m, \nonumber\\
		&\wp(\alpha_j) = k_j^2 - 2M,\ \ \ \ \ \wp_t(\alpha_j) = 4\ \ \ (j = 1,2,3), \nonumber
	\end{align}
	and
	\begin{equation}
		Y_j^{(2)} = \frac{\sigma(t + \beta - \alpha_j)}{\sigma(t + \beta)}e^{(\zeta(\alpha_j) + k_j)t}\ \ \ \ \ (j=1,2,3). \nonumber
	\end{equation}
	Set \(D = \frac{1}{\sigma(\beta)}{\rm diag}(\sigma(\beta-\alpha_1), \sigma(\beta-\alpha_2), \sigma(\beta-\alpha_3))\) and
	\begin{align}
		K_1 = \left(\begin{array}{ccc}
			\frac{a(k-m-a^2)}{2} & 1 & -\frac{a(k+m+a^2)}{2} \\
			-\frac{1}{a(m-a^2)} & 1 & -\frac{1}{a(m-a^2)} \\
			-\frac{a(k+m+a^2)}{2} & 1 & \frac{a(k-m-a^2)}{2}
		\end{array}\right). \nonumber
	\end{align}
	Then,
	\begin{equation}\label{F}
		\tilde{\Phi}_+^{-1}(t,-i) = K_1^{-1}D^{-1}\left(\begin{array}{ccc}
			\frac{1}{4a}\{(Y_1^{(2)})_{tt} - \frac{\wp_t(t+\beta)}{\wp(t+\beta)-M}(Y_1^{(2)})_t\} & Y_1^{(2)} & -\frac{a}{2}(Y_1^{(2)})_t \\
			\frac{1}{4a}\{(Y_2^{(2)})_{tt} - \frac{\wp_t(t+\beta)}{\wp(t+\beta)-M}(Y_2^{(2)})_t\} & Y_2^{(2)} & -\frac{a}{2}(Y_2^{(2)})_t \\
			\frac{1}{4a}\{(Y_3^{(2)})_{tt} - \frac{\wp_t(t+\beta)}{\wp(t+\beta)-M}(Y_3^{(2)})_t\} & Y_3^{(2)} & -\frac{a}{2}(Y_3^{(2)})_t \\
		\end{array}\right)
	\end{equation}
	is a solution to (\ref{A'}) on \(\mathbb{C} \backslash \tilde{\Lambda}\).
\end{prop}
\begin{proof}
	Since \(Y_1^{(2)}, Y_2^{(2)}, Y_3^{(2)}\) are solutions of (\ref{B}), the right hand side of (\ref{F}) satisfies (\ref{A'}). Using
	\begin{align}
		Y_j^{(2)}(0) &= \frac{\sigma(\beta - \alpha_j)}{\sigma(\beta)}, \nonumber\\
		(Y_j^{(2)})_t(0) &= \frac{\sigma(\beta - \alpha_j)}{\sigma(\beta)}\tilde{\psi}(0) =
		\frac{2(m + k_j + a^2)}{a^2(m^2 - a^4 - k_j^2)} \cdot \frac{\sigma(\beta - \alpha_j)}{\sigma(\beta)}, \nonumber\\
		(Y_j^{(2)})_{tt}(0) &= \frac{\sigma(\beta - \alpha_j)}{\sigma(\beta)}\{\tilde{\psi}'(0) + \tilde{\psi}(0)^2\} =
		\frac{4(m^2 - a^4 + (m + a^2)k_j)}{a^2(m^2 - a^4 - k_j^2)} \cdot \frac{\sigma(\beta - \alpha_j)}{\sigma(\beta)}. \nonumber
	\end{align}
	we obtain
	\begin{align}
		&\left(\begin{array}{ccc}
			\frac{1}{4a}\{(Y_1^{(2)})_{tt}(0) - \frac{\wp_t(\beta)}{\wp(\beta)-M}(Y_1^{(2)})_t(0)\} & Y_1^{(2)}(0) & -\frac{a}{2}(Y_1^{(2)})_t(0) \\
			\frac{1}{4a}\{(Y_2^{(2)})_{tt}(0) - \frac{\wp_t(\beta)}{\wp(\beta)-M}(Y_2^{(2)})_t(0)\} & Y_2^{(2)}(0) & -\frac{a}{2}(Y_2^{(2)})_t(0) \\
			\frac{1}{4a}\{(Y_3^{(2)})_{tt}(0) - \frac{\wp_t(\beta)}{\wp(\beta)-M}(Y_3^{(2)})_t(0)\} & Y_3^{(2)}(0) & -\frac{a}{2}(Y_3^{(2)})_t(0) \\
		\end{array}\right) = DK_1, \nonumber
	\end{align}
	and hence the right-hand side of \((\ref{F})\) evaluated at \(t=0\) is equal to \(I_3\). Therefore, the result follows from the uniqueness of the corresponding initial value problem.
\end{proof}\vskip\baselineskip

\begin{lem}
	Let \(\alpha\) satisfy \(\wp(\alpha) = M,\ \wp_t(\alpha) = 4\) and set \(\alpha_1 = \alpha_3 = \alpha\) and \(\alpha_2 = -2\alpha\). Then \(\alpha_1,\alpha_2\) and \(\alpha_3\) satisfy the condition in Proposition \ref{prop4.2}.
\end{lem}
\begin{proof}
	By assumption, \(\alpha_1\) and \(\alpha_3\) satisfy the required conditions. Since \(\wp_{tt}(-\alpha) = 6\wp(-\alpha)^2 - 6M^2 = 0\), the duplication formula for the Weierstrass \(\wp\)-function gives
	\begin{equation}
		\wp(\alpha_2) = \wp(-2\alpha) = \frac{1}{4}\left\{\frac{\wp''(-\alpha)}{\wp'(-\alpha)}\right\}^2 - 2\wp(-\alpha) = -2\wp(\alpha) = -2M. \nonumber
	\end{equation}
	Moreover,
	\begin{align}
		\wp_t(\alpha_3) &= \wp_t(-2\alpha) = \frac{\wp_{tt}(-\alpha)\wp_{ttt}(-\alpha)}{4\{\wp_t(-\alpha)\}^2} - \frac{\{\wp_{tt}(-\alpha)\}^3}{4\{\wp_t(-\alpha)\}^3} - \wp_t(-\alpha) = -\wp_t(-\alpha) \nonumber\\
		&= 4. \nonumber
	\end{align}
	Hence \(\alpha_2\) also satisfies the conditions in Proposition \ref{prop4.2}.
\end{proof}\vskip\baselineskip

In the sequel,, we set \(\alpha_1 = \alpha_3 = \alpha\) and \(\alpha_2 = -2\alpha\).

\begin{prop}\label{prop4.3}
	Writing \(f = \begin{pmatrix}
		x_1 \\ x_2 \\ x_3
	\end{pmatrix}\)
	then, for \(x\in \mathbb{R} \backslash \tilde{\Lambda},\ y \in \mathbb{R}\) we have
	\begin{align}
		x_1 &=\text{\scriptsize\( A_1\frac{\sigma(x+\beta-\alpha)}{\sigma(x+\beta)}e^{\zeta(\alpha)x}\cos{(ky)} + B_1\frac{\sigma(x+\beta+2\alpha)}{\sigma(x+\beta)}e^{-2\zeta(\alpha)x}\)}, \nonumber\\
		x_2 &=\text{\scriptsize\( C_1\frac{\sigma(x+\beta-\alpha)}{\sigma(x+\beta)}e^{\zeta(\alpha)x}\cos{(ky)} + D_1\frac{\sigma(x+\beta-\alpha)}{\sigma(x+\beta)}e^{\zeta(\alpha)x}\sin{(ky)} + E_1\frac{\sigma(x+\beta+2\alpha)}{\sigma(x+\beta)}e^{-2\zeta(\alpha)x}\)}, \nonumber\\
		x_3 &=\text{\scriptsize\( C_1\frac{\sigma(x+\beta-\alpha)}{\sigma(x+\beta)}e^{\zeta(\alpha)x}\cos{(ky)} - D_1\frac{\sigma(x+\beta-\alpha)}{\sigma(x+\beta)}e^{\zeta(\alpha)x}\sin{(ky)} + E_1\frac{\sigma(x+\beta+2\alpha)}{\sigma(x+\beta)}e^{-2\zeta(\alpha)x}\)}. \nonumber
	\end{align}
	\normalsize
	where
	\begin{align}
		A_1 &= \text{\scriptsize\( -\frac{2(m-a^2+a^{-1})\sigma(\beta)}{\sqrt{3}ak^2\sigma(\beta-\alpha)}\)}, \ \ \ B_1 =\text{\scriptsize\( \frac{(m-a^2)(m+a^2+2a^{-1})\sigma(\beta)}{\sqrt{3}k^2\sigma(\beta+2\alpha)}\)}, \nonumber\\
		C_1 &= \text{\scriptsize\( \frac{(m-a^2-2a^{-1})\sigma(\beta)}{\sqrt{3}ak^2\sigma(\beta-\alpha)}\)}, \ \ \ D_1 = \frac{\sigma(\beta)}{ak\sigma(\beta-\alpha)},\ \ \
		E_1 = \text{\scriptsize\( \frac{(m-a^2)(m+a^2-a^{-1})\sigma(\beta)}{\sqrt{3}k^2\sigma(\beta+2\alpha)}\)}, \nonumber
	\end{align}
\end{prop}
\begin{proof}
	Since
	\begin{align}
		\zeta(\alpha_2) &= \zeta(-2\alpha) = \zeta(-\alpha) - \zeta(\alpha) -\frac{1}{2}\frac{\wp_t(-2\alpha) - \wp_t(\alpha)}{\wp(-2\alpha) - \wp(\alpha)} = -2\zeta(\alpha). \nonumber
	\end{align}
	we obtain
	\begin{align}
		D^{-1}\begin{pmatrix}
			Y_1^{(2)}(x) \\ Y_2^{(2)}(x) \\ Y_3^{(2)}(x)
		\end{pmatrix}
		&= \begin{pmatrix}
			\frac{\sigma(\beta)\sigma(x + \beta - \alpha)}{\sigma(\beta-\alpha)\sigma(x + \beta)}e^{(\zeta(\alpha) + k)x} \\
			\frac{\sigma(\beta)\sigma(x + \beta + 2\alpha)}{\sigma(\beta + 2\alpha)\sigma(x + \beta)}e^{-2\zeta(\alpha)x} \\
			\frac{\sigma(\beta)\sigma(x + \beta - \alpha)}{\sigma(\beta - \alpha)\sigma(x + \beta)}e^{(\zeta(\alpha) - k)x}
		\end{pmatrix}. \nonumber
	\end{align}
	
	It follows from
	\begin{align}
		K_1\left(\begin{array}{ccc}
			m & -a^{-1} & a^2 \\
			a^{-1} & 0 & -a^{-1} \\
			-a^2 & a^{-1} & -m
		\end{array}\right)K_1^{-1} = \left(\begin{array}{ccc}
			-k & 0  & 0  \\
			0 	& 0 & 0  \\
			0 	& 0  & k
		\end{array}\right) \nonumber
	\end{align}
	and
	\begin{align}
		&TK_1^{-1} \nonumber\\
		&= \text{\tiny \(-\frac{1}{2\sqrt{3}ak^2}\left(\begin{array}{ccc}
				2(m-a^2+a^{-1}) & -2a(m-a^2)(m+a^2+2a^{-1}) & 2(m-a^2+a^{-1}) \\
				-m+a^2+2a^{-1} - i\sqrt{3}k & -2a(m-a^2)(m+a^2-a^{-1}) & 	-m+a^2+2a^{-1} + i\sqrt{3}k \\
				-m+a^2+2a^{-1} + i\sqrt{3}k & -2a(m-a^2)(m+a^2-a^{-1}) & 	-m+a^2+2a^{-1} - i\sqrt{3}k
			\end{array}\right)\)}, \nonumber
	\end{align}
	that
	\begin{align}
		f &= TK_1^{-1}\exp{\left(tK_1\left(\begin{array}{ccc}
				m & -a^{-1} & a^2 \\
				a^{-1} & 0 & -a^{-1} \\
				-a^2 & a^{-1} & -m
			\end{array}\right)K_1^{-1}\right)}D^{-1}\begin{pmatrix}
			Y_1^{(2)}(x) \\ Y_2^{(2)}(x) \\ Y_3^{(2)}(x)
		\end{pmatrix} \nonumber\\
		&= TK_1^{-1}\left(\begin{array}{ccc}
			e^{-kt} & 0 & 0 \\
			0 & 1 & 0 \\
			0 & 0 & e^{kt}
		\end{array}\right)D^{-1}\begin{pmatrix}
			Y_1^{(2)}(x) \\ Y_2^{(2)}(x) \\ Y_3^{(2)}(x)
		\end{pmatrix} \nonumber\\
		&= \text{\tiny \(\begin{pmatrix}
				A_1\frac{\sigma(x+\beta-\alpha)}{\sigma(x+\beta)}e^{\zeta(\alpha)x}\cos{(ky)} + B_1\frac{\sigma(x+\beta+2\alpha)}{\sigma(x+\beta)}e^{-2\zeta(\alpha)x}\\ C_1\frac{\sigma(x+\beta-\alpha)}{\sigma(x+\beta)}e^{\zeta(\alpha)x}\cos{(ky)} + D_1\frac{\sigma(x+\beta-\alpha)}{\sigma(x+\beta)}e^{\zeta(\alpha)x}\sin{(ky)} + E_1\frac{\sigma(x+\beta+2\alpha)}{\sigma(x+\beta)}e^{-2\zeta(\alpha)x}\\
				C_1\frac{\sigma(x+\beta-\alpha)}{\sigma(x+\beta)}e^{\zeta(\alpha)x}\cos{(ky)} - D_1\frac{\sigma(x+\beta-\alpha)}{\sigma(x+\beta)}e^{\zeta(\alpha)x}\sin{(ky)} + E_1\frac{\sigma(x+\beta+2\alpha)}{\sigma(x+\beta)}e^{-2\zeta(\alpha)x}
			\end{pmatrix}
			\)}. \nonumber
	\end{align}
\end{proof}\vskip\baselineskip

\subsubsection{The case \(M = 0\)}
Assume that \(M = \frac{1}{3}(m^2 - a^4 - 2a^{-2}) = 0\), then \(k_1 = k_2 = k_3 = M = 0\), and (\ref{B}) reduces to
\begin{equation}\label{B'}
	Y_{ttt} - 6\wp(t+\beta)Y_t = 8\lambda^3iY. \tag{4'}
\end{equation}

\begin{lem}\label{lem4.4}
	Let \(\alpha, \beta\) satisfy
	\begin{align}
		&\wp(\beta) = 2a^{-2},\ \ \ \ \ \wp_t(\beta) =  4a^{-2}m, \nonumber\\
		&\wp(\alpha) = 0,\ \ \ \ \ \wp_t(\alpha) = 4 \nonumber
	\end{align}
	and
	\begin{equation}
		q(t) = \int_{0}^{t}(t-s)\left(\frac{\sigma(s+\beta)}{\sigma(s+\beta-\alpha)}\right)^3e^{-3\zeta(\alpha)s}ds. \nonumber
	\end{equation}
	Then \(Y_1^{(2)} = \frac{\sigma(t+\beta-\alpha)}{\sigma(t+\beta)}e^{\zeta(\alpha)t}, Y^{(2)}_2 = tY_1^{(2)}(t)\) and \(Y_3^{(2)} = q(t)Y_1^{(2)}(t)\) are fundamental solutions to (\ref{B'}) on \(\mathbb{C} \backslash \tilde{\Lambda}\).
\end{lem}
\begin{proof}
	As shown in Section \ref{sec4.1}, \(Y_2^{(2)}\) is a solution of (\ref{B'}). Since \((Y_1^{(2)})_t = \tilde{\psi}Y_1^{(2)}\) and \((Y_1^{(2)})_{tt} = 2\wp Y_1^{(2)}\), we obtain
	\begin{align}
		(Y_2^{(2)})_t &= Y_1^{(2)}+ t(Y_1^{(2)})_t,\ \ \ \ \ (Y_2^{(2)})_{tt} = 2(Y_1^{(2)})_t + t(Y_1^{(2)})_{tt}, \nonumber\\
		(Y_2^{(2)})_{ttt} &= 3(Y_1^{(2)})_{tt} + t(Y_1^{(2)})_{ttt} = 6\wp Y_1^{(2)} + 6t\wp (Y_1^{(2)})_t + 8\lambda^3 itY_1^{(2)} \nonumber\\
		&= 6\wp (Y_2^{(2)})_t + 8\lambda^3 i Y_2^{(2)}. \nonumber
	\end{align}
	Moreover, since
	\begin{align}
		&q_t = \int_{0}^{t}\left(\frac{\sigma(s+\beta)}{\sigma(s+\beta-\alpha)}\right)^3e^{-3\zeta(\alpha)s}ds,\ \ \ \ \ q_{tt} = \left(\frac{\sigma(t+\beta)}{\sigma(t+\beta-\alpha)}\right)^3e^{-3\zeta(\alpha)t}, \nonumber\\
		&q_{ttt} = -3\tilde{\psi}(t)q_{tt}, \nonumber
	\end{align}
	we obtain
	\begin{align}
		(Y_3^{(2)})_t &= q_tY_1^{(2)} + q(Y_1^{(2)})_t,\ \ \ \ \ 	(Y_3^{(2)})_{tt} = q_{tt}Y_1^{(2)} + 2q_t(Y_1^{(2)})_t + q(Y_1^{(2)})_{tt}, \nonumber\\
		(Y_3^{(2)})_{ttt} &= q_{ttt}Y_1^{(2)} + 3q_{tt}(Y_1^{(2)})_t + 3q_t(Y_1^{(2)})_{tt} + q(Y_1^{(2)})_{ttt} \nonumber\\
		&= 6q_t\wp Y_1^{(2)} + 6q\wp (Y_1^{(2)})_t + 8q\lambda^3 iY_1^{(2)} \nonumber\\
		&= 6\wp(Y_3^{(2)})_t + 8\lambda^3 i Y_3^{(2)}. \nonumber
	\end{align}
\end{proof}\vskip\baselineskip

\begin{prop}
	Set
	\begin{equation}
		K_2 = \frac{\sigma(\beta-\alpha)}{\sigma(\beta)}\left(\begin{array}{ccc}
			-\frac{a}{2}(m+a^2) & 1 & -\frac{a}{2}(m+a^2) \\
			\frac{a}{2} & 0 & -\frac{a}{2} \\
			\frac{1}{4a}\left(\frac{\sigma(\beta)}{\sigma(\beta - \alpha)}\right)^3 & 0 & 0
		\end{array}\right). \nonumber
	\end{equation}
	Then,
	\begin{equation}\label{G}
		\tilde{\Phi}_+^{-1}(t,-i) = K_2^{-1}\left(\begin{array}{ccc}
			\frac{1}{4a}\{(Y_1^{(2)})_{tt} -
			\frac{\wp_t(t+\beta)}{\wp(t+\beta)}(Y_1^{(2)})_t\} & Y_1^{(2)} & -\frac{a}{2}(Y_1^{(2)})_t \\
			\frac{1}{4a}\{(Y_2^{(2)})_{tt} -
			\frac{\wp_t(t+\beta)}{\wp(t+\beta)}(Y_2^{(2)})_t\} & Y_2^{(2)} & -\frac{a}{2}(Y_2^{(2)})_t \\
			\frac{1}{4a}\{(Y_3^{(2)})_{tt} -
			\frac{\wp_t(t+\beta)}{\wp(t+\beta)}(Y_3^{(2)})_t\} & Y_3^{(2)} & -\frac{a}{2}(Y_3^{(2)})_t \\
		\end{array}\right),
	\end{equation}
	is a solution to (\ref{A'}) on \(\mathbb{C} \backslash \tilde{\Lambda}\).
\end{prop}
\begin{proof}
	From Lemma \ref{lem4.4}, the right hand side of (\ref{G}) satisfies (\ref{A'}). Using
	\begin{align}
		&Y_1^{(2)}(0) = \frac{\sigma(\beta-\alpha)}{\sigma(\beta)},\  (Y_1^{(2)})_t(0) = (m+a^2)\frac{\sigma(\beta-\alpha)}{\sigma(\beta)}, \ (Y_1^{(2)})_{tt}(0) = \frac{4}{a^2}\frac{\sigma(\beta-\alpha)}{\sigma(\beta)}, \nonumber\\
		&Y_2^{(2)}(0) = 0,\ \ \ \ \ (Y_2^{(2)})_t(0) = \frac{\sigma(\beta-\alpha)}{\sigma(\beta)},\ \ \ \ \ \ (Y_2^{(2)})_{tt}(0) = 2(m+a^2)\frac{\sigma(\beta-\alpha)}{\sigma(\beta)}, \nonumber\\
		&Y_3^{(2)}(0) = 0,\ \ \ \ \ (Y_3^{(2)})_t(0) = 0,\ \ \ \ \ (Y_3^{(2)})_{tt}(0) = \left(\frac{\sigma(\beta)}{\sigma(\beta - \alpha)}\right)^2, \nonumber
	\end{align}
	we obtain
	\begin{align}
		&\left(\begin{array}{ccc}
			\frac{1}{4a}\{(Y_1^{(2)})_{tt}(0) - \frac{\wp_t(\beta)}{\wp(\beta)}(Y_1^{(2)})_t(0)\} & Y_1^{(2)}(0) & -\frac{a}{2}(Y_1^{(2)})_t(0) \\
			\frac{1}{4a}\{(Y_2^{(2)})_{tt}(0) - \frac{\wp_t(\beta)}{\wp(\beta)}(Y_2^{(2)})_t(0)\} & Y_2^{(2)}(0) & -\frac{a}{2}(Y_2^{(2)})_t(0) \\
			\frac{1}{4a}\{(Y_3^{(2)})_{tt}(0) - \frac{\wp_t(\beta)}{\wp(\beta)}(Y_3^{(2)})_t(0)\} & Y_3^{(2)}(0) & -\frac{a}{2}(Y^{(2)}_3)_t(0) \\
		\end{array}\right) = K_2. \nonumber
	\end{align}
	Evaluating the right-hand side of (\ref{G}) at \(t=0\), we obtain \(I_3\). The conclusion then follows from uniqueness.
\end{proof}\vskip\baselineskip

\begin{prop}\label{prop4.5}
	Writing \(f = \begin{pmatrix}
		x_1 \\ x_2 \\ x_3
	\end{pmatrix}\)
	then, for \(x\in \mathbb{R} \backslash \tilde{\Lambda},\ y \in \mathbb{R}\) we have
	\begin{align}
		x_1 &= \frac{\sigma(x+\beta-\alpha)}{\sigma(x+\beta)}e^{\zeta(\alpha)x}\left\{A_2(x^2+y^2) + B_2(x)\right\}, \nonumber\\ 
		x_2 &= \frac{\sigma(x+\beta-\alpha)}{\sigma(x+\beta)}e^{\zeta(\alpha)x}\left\{C_2(x^2+y^2) + D_2y + E_2(x)\right\}, \nonumber\\
		x_3 &= \frac{\sigma(x+\beta-\alpha)}{\sigma(x+\beta)}e^{\zeta(\alpha)x}\left\{C_2(x^2+y^2) - D_2y + E_2(x)\right\}, \nonumber
	\end{align}
	where
	\begin{align}
		A_2 &= \frac{(m+a^2+2a^{-1})(m-a^2)\sigma(\beta)}{2\sqrt{3}\sigma(\beta-\alpha)},\ \ \ C_2 = \frac{(m+a^2-a^{-1})(m-a^2)\sigma(\beta)}{2\sqrt{3}\sigma(\beta-\alpha)},  \nonumber\\
		D_2 &= \text{\small \(\frac{\sigma(\beta)}{a\sigma(\beta-\alpha)}\)}, \nonumber\\
		B_2(x) &= \text{\footnotesize \(\frac{\sigma(\beta)}{\sqrt{3}\sigma(\beta-\alpha)}\left\{
			1 - (m+a^2+2a^{-1})x + 4a^2(m+a^2+2a^{-1})q(x)\left(\frac{\sigma(\beta-\alpha)}{\sigma(\beta)}\right)^3
			\right\}\)} \nonumber\\
		E_2(x) &= \text{\footnotesize \(\frac{\sigma(\beta)}{\sqrt{3}\sigma(\beta-\alpha)}\left\{
			1 - (m+a^2-a^{-1})x + 4a^2(m+a^2-a^{-1})q(x)\left(\frac{\sigma(\beta-\alpha)}{\sigma(\beta)}\right)^3
			\right\}\)}. \nonumber
	\end{align}
\end{prop}
\begin{proof}
	Let
	\begin{equation}
		P = \left(\begin{array}{ccc}
			1 & 0 & 0 \\
			0 & i & 0 \\
			0 & 1 & -4a^2\left(\frac{\sigma(\beta-\alpha)}{\sigma(\beta)}\right)^3
		\end{array}\right). \nonumber
	\end{equation}
	Since
	\small
	\begin{align}
		K_2\left(\begin{array}{ccc}
			m & -a^{-1} & a^2 \\
			a^{-1} & 0 & -a^{-1} \\
			-a^2 & a^{-1} & -m
		\end{array}\right)K_2^{-1} &= \left(\begin{array}{ccc}
			0 & 0 & 0 \\
			-1 & 0 & 0 \\
			-\frac{1}{4a^2}\left(\frac{\sigma(\beta)}{\sigma(\beta-\alpha)}\right)^3 & \frac{m-a^2}{4a^2}\left(\frac{\sigma(\beta)}{\sigma(\beta-\alpha)}\right)^3 & 0
		\end{array}\right) \nonumber\\
		&= P^{-1}\left(\begin{array}{ccc}
			0 & 0 & 0 \\
			-i & 0 & 0 \\
			0 & i(m-a^2) & 0
		\end{array}\right)P \nonumber
	\end{align}
	\normalsize
	and
	\begin{equation}
		TK_2^{-1}P^{-1} = \frac{\sigma(\beta)}{\sqrt{3}\sigma(\beta-\alpha)}\left(\begin{array}{ccc}
			1 & 0 & -m-a^2-2a^{-1} \\
			1 & \sqrt{3}a^{-1} & -m-a^2+a^{-1} \\
			1 & -\sqrt{3}a^{-1} & -m-a^2+a^{-1}
		\end{array}\right), \nonumber
	\end{equation}
	we obtain
	\begin{align}
		f &= TK_2^{-1}\exp{\left(tK_2\left(\begin{array}{ccc}
				m & -a^{-1} & a^2 \\
				a^{-1} & 0 & -a^{-1} \\
				-a^2 & a^{-1} & -m
			\end{array}\right)K_2^{-1}\right)}\begin{pmatrix}
			Y_1^{(2)}(x) \\ Y_2^{(2)}(x) \\ Y_3^{(2)}(x)
		\end{pmatrix} \nonumber\\
		&= TK_2^{-1}P^{-1}\left(\begin{array}{ccc}
			1 & 0 & 0 \\
			-it & 1 & 0 \\
			\frac{(m-a^2)}{2}t^2 & i(m-a^2)t & 1
		\end{array}\right)P\begin{pmatrix}
			Y_1^{(2)}(x) \\ Y_2^{(2)}(x) \\ Y_3^{(2)}(x)
		\end{pmatrix} \nonumber\\
		&= \begin{pmatrix}
			\frac{\sigma(x+\beta-\alpha)}{\sigma(x+\beta)}e^{\zeta(\alpha)x}\left\{A_2(x^2+y^2) + B_2(x)\right\} \\ 
			\frac{\sigma(x+\beta-\alpha)}{\sigma(x+\beta)}e^{\zeta(\alpha)x}\left\{C_2(x^2+y^2) + D_2y + E_2(x)\right\} \\
			\frac{\sigma(x+\beta-\alpha)}{\sigma(x+\beta)}e^{\zeta(\alpha)x}\left\{C_2(x^2+y^2) - D_2y + E_2(x)\right\}
		\end{pmatrix}. \nonumber
	\end{align}
\end{proof}\vskip\baselineskip

\subsection{Classification of translationally equivariant affine spheres}
From Proposition \ref{prop4.3} and \ref{prop4.5}, we classify translationally equivariant hyperbolic affine spheres according to the sign of \(M = \frac{1}{3}(m^2 - a^2 - 2a^{-2})\).
\begin{theorem}\label{thm4.6}
	Let \(f\) be a hyperbolic affine sphere obtained from the Delaunay-type potential.
	\begin{itemize}
		\item [(i)] If \(3M = m^2 -2a^{-2} -a^4 < 0\), then
		\begin{align}
			\gamma_1 := \frac{i}{k}\left(\begin{array}{ccc}
				a^{-1} & 0  & 0  \\
				0 & m-a^2 & 0  \\
				0 	& 0  & a^{-1}
			\end{array}\right)K_1T^{-1} \in {\rm SL}_3 \mathbb{R} \nonumber
		\end{align}
		and \(\gamma_1f = \begin{pmatrix}
			X_1(x,y) \\ X_2(x,y) \\ X_3(x,y)
		\end{pmatrix}\) satisfies
		\begin{align}
			&X_1X_3 = -\frac{1}{a^2k^2}\left(\frac{\sigma(\beta)\sigma(x+\beta-\alpha)}{\sigma(x+\beta)\sigma(\beta-\alpha)}e^{\zeta(\alpha)x}\right)^2, \nonumber\\ 
			&X_2 = -\frac{i(m-a^2)\sigma(\beta)\sigma(x+\beta+2\alpha)}{k\sigma(x+\beta)\sigma(\beta+2\alpha)}e^{-2\zeta(\alpha)x}. \nonumber
		\end{align}
		The image is shown in Figure~1.
		
		\item [(ii)] If \(3M = m^2 -2a^{-2} -a^4 > 0\), then
		\begin{align}
			\gamma_2 := \frac{1}{2k}\left(\begin{array}{ccc}
				2a^{-1} & 0 & 2a^{-1} \\
				0 & m-a^2 & 0 \\
				2ia^{-1} & 0 & -2ia^{-1}
			\end{array}\right)K_1T^{-1} \in {\rm SL}_3 \mathbb{R} \nonumber
		\end{align}
		and \(\gamma_2f = \begin{pmatrix}
			X_1(x,y) \\ X_2(x,y) \\ X_3(x,y)
		\end{pmatrix}\) satisfies
		\begin{align}
			&X_1^2 + X_3^2 = \frac{4}{a^2k^2}\left(\frac{\sigma(\beta)\sigma(x+\beta-\alpha)}{\sigma(x+\beta)\sigma(\beta-\alpha)}e^{\zeta(\alpha)x}\right)^2, \nonumber\\ 
			&X_2 = \frac{(m-a^2)\sigma(\beta)\sigma(x+\beta+2\alpha)}{2k\sigma(x+\beta)\sigma(\beta+2\alpha)}e^{-2\zeta(\alpha)x}. \nonumber
		\end{align}
		The image is shown in Figure~2.
		
		\item [(iii)] If \(3M = m^2 -2a^{-2} -a^4 = 0\), then
		\begin{equation}
			\gamma_3 = -\frac{\sigma(\beta)}{\sigma(\beta-\alpha)}\left(\begin{array}{cccc}
				0 & 1 & 0 \\
				2a^{-2} & 0 & 0 \\
				0 & 0 & 1
			\end{array}\right)PK_2T^{-1} \in {\rm SL}_3 \mathbb{R} \nonumber
		\end{equation}
		and \(\gamma_3f = \begin{pmatrix}
			X_1(x,y) \\ X_2(x,y) \\ X_3(x,y)
		\end{pmatrix}\) satisfies
		\begin{align}
			X_3 = Q(x)X_1^2 + R(x),\ \ \ \ \
			X_2 = -2a^{-2}\frac{\sigma(\beta)\sigma(x+\beta-\alpha)}{\sigma(x+\beta)\sigma(\beta-\alpha)}e^{\zeta(\alpha)x}, \nonumber
		\end{align}
		where
		\begin{align}
			Q(x) &=  \frac{(m-a^2)\sigma(x+\beta)\sigma(\beta-\alpha)}{2\sigma(\beta)\sigma(x+\beta-\alpha)}e^{-\zeta(\alpha)x}, \nonumber\\
			R(x) &= \text{\small\( \frac{\sigma(\beta)\sigma(x+\beta-\alpha)}{\sigma(x+\beta)\sigma(\beta-\alpha)}e^{\zeta(\alpha)x}\left\{\frac{m-a^2}{2}x^2 - x + 4a^2q(x)\left(\frac{\sigma(\beta-\alpha)}{\sigma(\beta)}\right)^3 \right\}\)}. \nonumber
		\end{align}
		The image is shown in Figure~3.
	\end{itemize}
\end{theorem}
\begin{proof}
	(i) If \(M < 0\), then \(ik \in \mathbb{R}\). Since
	\scriptsize
	\begin{equation}
		K_1T^{-1} = -\frac{\sqrt{3}a}{6}\left(\begin{array}{ccc}
			2(m+a^2-a^{-1}) & -(m+a^2+2a^{-1}) + i\sqrt{3}k & -m-a^2-2a^{-1} - i\sqrt{3}k \\
			-\frac{2(m-a^2-2a^{-1})}{a(m-a^2)} & -\frac{2(m-a^2+a^{-1})}{a(m-a^2)} & -\frac{2(m-a^2+a^{-1})}{a(m-a^2)} \\
			2(m+a^2-a^{-1}) & -(m+a^2+2a^{-1}) - i\sqrt{3}k & -m-a^2-2a^{-1} + i\sqrt{3}k
		\end{array}\right) \nonumber
	\end{equation}
	\normalsize
	and \({\rm det}(K_1T^{-1}) = \frac{ik^3a^2}{m-a^2}\), we have \(\gamma_1 \in {\rm SL}_3 \mathbb{R}\) and
	\begin{align}
		\begin{pmatrix}
			X_1 \\ X_2 \\ X_3
		\end{pmatrix} &= \gamma_1f \nonumber\\
		&= \frac{i}{k}\left(\begin{array}{ccc}
			a^{-1} & 0  & 0  \\
			0 & m-a^2 & 0  \\
			0 	& 0  & a^{-1}
		\end{array}\right)\left(\begin{array}{ccc}
			e^{-kt} & 0 & 0 \\
			0 & 1 & 0 \\
			0 & 0 & e^{kt}
		\end{array}\right)D^{-1}\begin{pmatrix}
			Y_1^{(2)}(x) \\ Y_2^{(2)}(x) \\ Y_3^{(2)}(x)
		\end{pmatrix} \nonumber\\ 
		&= \frac{i\sigma(\beta)}{k\sigma(x+\beta)}\begin{pmatrix}
			a^{-1}\frac{\sigma(x+\beta-\alpha)}{\sigma(\beta-\alpha)}e^{\zeta(\alpha)x}e^{-iky} \\
			(m-a^2)\frac{\sigma(x+\beta+2\alpha)}{\sigma(\beta+2\alpha)}e^{-2\zeta(\alpha)x}\\
			a^{-1}\frac{\sigma(x+\beta-\alpha)}{\sigma(\beta-\alpha)}e^{\zeta(\alpha)x}e^{iky}
		\end{pmatrix}. \nonumber
	\end{align}
	Hence, we obtain (i).\vskip\baselineskip
	
	\noindent (ii) If \(M > 0\), then \(k \in \mathbb{R}\). Since
	\begin{align}
		&\left(\begin{array}{ccc}
			2a^{-1} & 0 & 2a^{-1} \\
			0 & m-a^2 & 0 \\
			2ia^{-1} & 0 & -2ia^{-1}
		\end{array}\right)K_1T^{-1} \nonumber\\
		&= -\frac{\sqrt{3}}{3}\left(\begin{array}{ccc}
			4(m+a^2-a^{-1}) & -2(m+a^2+2a^{-1}) & -2(m+a^2+2a^{-1}) \\
			-(m-a^2-2a^{-1}) & -(m-a^2+a^{-1}) & -(m-a^2+a^{-1}) \\
			0 & -2\sqrt{3}k & 2\sqrt{3}k
		\end{array}\right) \nonumber
	\end{align}
	and its determinant is \(8k^3\), we obtain \(\gamma_2 \in {\rm SL}_3 \mathbb{R}\) and
	\begin{align}
		\begin{pmatrix}
			X_1 \\ X_2 \\ X_3
		\end{pmatrix} &= \gamma_2f \nonumber\\
		&= \text{\small \(\frac{1}{2k}\left(\begin{array}{ccc}
				2a^{-1} & 0 & 2a^{-1} \\
				0 & m-a^2 & 0 \\
				2ia^{-1} & 0 & -2ia^{-1}
			\end{array}\right)\left(\begin{array}{ccc}
				e^{-kt} & 0 & 0 \\
				0 & 1 & 0 \\
				0 & 0 & e^{kt}
			\end{array}\right)D^{-1}\begin{pmatrix}
				Y_1^{(2)}(x) \\ Y_2^{(2)}(x) \\ Y_3^{(2)}(x)
			\end{pmatrix}\)} \nonumber\\  
		&= \frac{\sigma(\beta)}{k\sigma(x+\beta)}\begin{pmatrix}
			2a^{-1}\frac{\sigma(x+\beta-\alpha)}{\sigma(\beta-\alpha)}e^{\zeta(\alpha)x}\cos{(ky)}\\
			(m-a^2)\frac{\sigma(x+\beta+2\alpha)}{2\sigma(\beta+2\alpha)}e^{-2\zeta(\alpha)x}\\
			2a^{-1}\frac{\sigma(x+\beta-\alpha)}{\sigma(\beta-\alpha)}e^{\zeta(\alpha)x}\sin{(ky)}
		\end{pmatrix}. \nonumber
	\end{align}
	Hence, we obtain (ii).\vskip\baselineskip
	
	\noindent (iii) If \(M = 0\), then \(k=0\). Since
	\begin{equation}
		PK_2T^{-1} =
		\text{\footnotesize\( \frac{\sqrt{3}a\sigma(\beta-\alpha)}{6\sigma(\beta)}\left(\begin{array}{ccc}
				-2(m+a^2-a^{-1}) & m+a^2+2a^{-1} & m+a^2+2a^{-1}\\
				0 & \sqrt{3} & -\sqrt{3} \\
				-2 & 1 & 1
			\end{array}\right)\)} \nonumber
	\end{equation}
	and \({\rm det}(PK_2T^{-1}) = \frac{a^2}{2}\left(\frac{\sigma(\beta-\alpha)}{\sigma(\beta)}\right)^3\), we obtain \(\gamma_3 \in {\rm SL}_3 \mathbb{R}\) and
	\begin{align}
		\begin{pmatrix}
			X_1 \\ X_2 \\ X_3
		\end{pmatrix} &= \gamma_3f \nonumber\\
		&= \text{\footnotesize \(-\frac{\sigma(\beta)}{\sigma(\beta-\alpha)}\left(\begin{array}{cccc}
				0 & 1 & 0 \\
				2a^{-2} & 0 & 0 \\
				0 & 0 & 1
			\end{array}\right)\left(\begin{array}{ccc}
				1 & 0 & 0 \\
				-it & 1 & 0 \\
				\frac{(m-a^2)}{2}t^2 & i(m-a^2)t & 1
			\end{array}\right)P\begin{pmatrix}
				Y_1^{(2)}(x) \\ Y_2^{(2)}(x) \\ Y_3^{(2)}(x)
			\end{pmatrix}\)} \nonumber\\
		&=
		-\text{\scriptsize\( \frac{\sigma(\beta)\sigma(x+\beta-\alpha)}{\sigma(x+\beta)\sigma(\beta-\alpha)}e^{\zeta(\alpha)x}\begin{pmatrix}
				y \\
				2a^{-2} \\
				-\frac{m-a^2}{2}y^2 -\frac{m-a^2}{2}x^2 + x -4a^2q(x)\left(\frac{\sigma(\beta-\alpha)}{\sigma(\beta)}\right)^3
			\end{pmatrix}\)}. \nonumber
	\end{align}
	Hence, we obtain (iii).
\end{proof}\vskip\baselineskip

\begin{figure}[H]
	\centering
	
	\begin{minipage}{0.6\textwidth}
		\centering
		\begin{tikzpicture}
			\useasboundingbox (-2.5,-2.5) rectangle (4,2.5);
			
			\draw[white] (-4,-3) grid (5,3);
			
			\draw[->] (-4,0)--(4,0);
			\draw[->] (0,-2.5)--(0,2.5);
			
			\draw (1,0)--(1,0.8);
			\draw (0,0.8)--(1,0.8);
			\draw (-1,0)--(-1,-0.8);
			\draw (0,-0.8)--(-1,-0.8);
			
			\draw[thick,domain = 0.3:3, smooth] plot(\x,{0.8/\x});
			\draw[thick,domain = -0.3:-3, smooth] plot(\x,{0.8/\x});
			
			\draw (4,0) node[right] {\(X_1\)};
			\draw (0,2.5) node[above] {\(X_3\)};
			
			\draw (1,0) node[below] {\(1\)};
			\draw (0,0.8) node[left] {\small \(-\frac{1}{a^2k^2}\left(\frac{\sigma(\beta)\sigma(x+\beta-\alpha)}{\sigma(x+\beta)\sigma(\beta-\alpha)}e^{\zeta(\alpha)x}\right)^2\)};
			
			\draw (-1,0) node[above] {\(-1\)};
			\draw (0,-0.8) node[right] {\small \(\frac{1}{a^2k^2}\left(\frac{\sigma(\beta)\sigma(x+\beta-\alpha)}{\sigma(x+\beta)\sigma(\beta-\alpha)}e^{\zeta(\alpha)x}\right)^2\)};
		\end{tikzpicture}
		\caption{A slice of \(\gamma_1f\) for \(M<0\)}
	\end{minipage}
\end{figure}


\begin{figure}[H]
	\centering
	
	\begin{minipage}{0.6\textwidth}
		\centering
		\begin{tikzpicture}
			\useasboundingbox (-2.5,-2.5) rectangle (4,2.5);
			
			\draw[white] (-4,-3) grid (5,3);
			
			\draw[->] (-4,0)--(4,0);
			\draw[->] (0,-2.5)--(0,2.5);
			
			\draw (2,0)--(2,0.2);
			
			\draw[thick] (0,0) circle[radius=2];
			
			\draw (4,0) node[right] {\(X_1\)};
			\draw (0,2.5) node[above] {\(X_3\)};
			
			\draw (1.9,0.5) node[right] {\small \(\frac{2\sigma(\beta)\sigma(x+\beta-\alpha)}{a|k|\sigma(x+\beta)\sigma(\beta-\alpha)}e^{\zeta(\alpha)x}\)};
		\end{tikzpicture}
		\caption{A slice of \(\gamma_2f\) for \(M>0\)}
	\end{minipage}
\end{figure}

\begin{figure}[H]
	\centering
	
	\begin{minipage}{0.6\textwidth}
		\centering
		\begin{tikzpicture}
			\useasboundingbox (-2.5,-2.5) rectangle (4,2.5);
			
			\draw[white] (-4,-3) grid (5,3);
			
			\draw[->] (-4,0)--(4,0);
			\draw[->] (0,-2.5)--(0,2.5);
			
			\draw (1,0)--(1,1);
			\draw (0,1)--(1,1);
			
			\draw[thick,domain = -2:2, smooth] plot(\x,{0.5*\x*\x+0.5});
			
			\draw (4,0) node[right] {\(X_1\)};
			\draw (0,2.5) node[above] {\(X_3\)};
			
			\draw (1,0) node[below] {\(1\)};
			\draw (0,0.3) node[right] {\small \(R\)};
			
			\draw (0.08,1.3) node[left] {\footnotesize \(Q+R\)};
		\end{tikzpicture}
		\caption{A slice of \(\gamma_3f\) for \(M=0\)}
	\end{minipage}
\end{figure}

As a consequence, we obtain the following classification result.
\begin{cor}\label{cor4.7}
	Every translationally invariant solution of the Tzitz\'eica equation gives rise to a hyperbolic affine sphere equiaffinely equivalent to an affine sphere whose slice curve is a circle, a hyperbola, or a parabola.
\end{cor}
\begin{proof}
	Choose a conformal coordinate \(t=x+iy\) and a cubic differential so that the conformal factor depends only on \(x\) and the resulting Maurer--Cartan form is precisely the one generated by a Delaunay-type potential with \(\lambda=-i\). By the fundamental theorem of affine differential geometry, affine spheres with the same Maurer--Cartan form are equiaffinely equivalent. The conclusion now follows from Theorem \ref{thm4.6}.
\end{proof}\vskip\baselineskip

This classification agrees with the Calabi conjecture for affine spheres \cite{C1972}. the Calabi-Cheng-Yau theorem on affine spheres establishes a one-to-one correspondence between hyperbolic affine spheres and proper convex cones in \(\mathbb{R}^3\) \cite{CY1986}. In particular, conic sections are classified into ellipses, hyperbolas, and parabolas. Hence, our result can be viewed as a concrete realization of this correspondence in the translationally equivariant setting.

	\section*{Acknowledgement}
	Dedicated to my father, Professor Seiichi Udagawa, on the occasion of his retirement. This paper is a part of the outcome of research performed under a Waseda University Grant for Special Research Projects (Project number: 2026C-087).
	
	\section*{Conflict of interests}
	The author has no conflicts to disclose.

\bibliography{mybibfile}

\begin{thebibliography}{10}

\bibitem{BD2001}
V.~Balan and J.~Dorfmeister.
\newblock {Birkhoff decomposition and Iwasawa decomposition for loop groups}.
\newblock {\em Tohoku Math. J.}, 53:593--615, 2001.

\bibitem{BRS2010}
D.~Brander, W.~Rossman, and N.~Schmitt.
\newblock {Holomorphic representation of constant mean curvature surfaces in
  Minkowski space: consequences of non-compactness in loop group methods}.
\newblock {\em Adv. Math.}, 223(3):949--986, 2010.

\bibitem{C1972}
E.~Calabi.
\newblock {Complete affine hyperspheres. I}.
\newblock In {\em {Eugenio Calabi---Collected Works}}, pages 453--472.
  Springer, 2020.

\bibitem{CY1986}
S.-Y. Cheng and S.-T. Yau.
\newblock {Complete affine hypersurfaces. Part I. The completeness of affine
  metrics}.
\newblock In {\em {Selected Works of Shing-Tung Yau. Part 1. 1971--1991. Vol.
  2. Metric geometry and harmonic functions}}, pages 13--40. Int. Press, 2019.

\bibitem{DE2001}
J.~Dorfmeister and U.~Eitner.
\newblock {Weierstra{\ss}-type representation of affine spheres}.
\newblock {\em Abh. Math. Sem. Univ. Hamburg}, 71:225--250, 2001.

\bibitem{DGR2010}
J.~Dorfmeister, M.~Guest, and W.~Rossman.
\newblock {The tt* structure of the quantum cohomology of $\mathbb C$$P^1$ from
  the viewpoint of differential geometry}.
\newblock {\em Asian J. Math.}, 14(3):417--437, 2010.

\bibitem{DM2016}
J.~Dorfmeister and H.~Ma.
\newblock {Explicit expressions for the Iwasawa factors, the metric and the
  monodromy matrices for minimal Lagrangian surfaces in $\mathbb{CP}^2$}.
\newblock In {\em {Dynamical Systems, Number Theory and Applications}}, pages
  19--47. World Sci. Publ., 2016.

\bibitem{DPW1998}
J.~Dorfmeister, F.~Pedit, and H.~Wu.
\newblock Weierstrass type representation of harmonic maps into symmetric
  spaces.
\newblock {\em Comm. Anal. Geom.}, 6:633--668, 1998.

\bibitem{F1991}
O.~Forster.
\newblock {\em {Lectures on Riemann surfaces, Translated from the 1997 German
  original by Bruce Gilligan. Reprint of the 1981 English translation}}.
\newblock Grad. Texts in Math. Springer-Verlag, 1991.

\bibitem{H2014}
R.~Hildebrand.
\newblock Analytic formulas for complete hyperbolic affine spheres.
\newblock {\em Beitr. Algebra Geom.}, 55(2):497--520, 2014.

\bibitem{H2022}
R.~Hildebrand.
\newblock {Self-associated three-dimensional cones}.
\newblock {\em Beitr. Algebra Geom.}, 63(4):867--906, 2022.

\bibitem{LW2016}
Z.~Lin and E.~Wang.
\newblock {The associated families of semi-homogeneous complete hyperbolic
  affine spheres}.
\newblock {\em Acta Math. Sci. Ser. B (Engl. Ed.)}, 36(3):765--781, 2016.

\bibitem{SW1993}
U.~Simon and C.~P. Wang.
\newblock {Local theory of affine 2-spheres}.
\newblock In {\em {Differential Geometry: Riemannian Geometry (Los Angeles, CA,
  1990)}}, volume~54 of {\em Proc. Sympos. Pure Math.}, pages 585--598. Amer.
  Math. Soc., Providence, RI, 1993.

\bibitem{U2026}
T.~Udagawa.
\newblock {The Iwasawa factorization with rotationally symmetric parts and the
  Lam\'e equation}.
\newblock {\em Tohoku Math. J.}, to appear.

\bibitem{U2024}
T.~Udagawa.
\newblock {Globality of the DPW construction for Smyth potentials in the case
  of ${\rm SU}_{1,1}$}.
\newblock {\em Differ. Geom. Appl.}, 97:\ Paper No. 102211, 30 pp, 2024.

\end{thebibliography}
\bibliographystyle{plain}

	\em
	\noindent
	Department of Applied Mathematics\newline
	Faculty of Science and Engineering\newline
	Waseda University\newline
	3-4-1 Okubo, Shinjuku, Tokyo 169-8555\newline
	JAPAN

\end{document}